\documentclass{sadhana}
\usepackage{subfigure,enumitem,txfonts}
\usepackage{graphicx,amsmath,bm}
\newtheorem{theorem}{Theorem}[section]
\newtheorem{lemma}[theorem]{Lemma}

\newtheorem{cor}[theorem]{Corollary}
\newtheorem{definition}[theorem]{Definition}
\newtheorem{example}[theorem]{Example}

\newtheorem{remark}[theorem]{Remark}
\newenvironment{proof}{{\noindent\it Proof}\quad}{\hfill $\square$\par}
\numberwithin{equation}{section}

\begin{document}

\title{Fritz-John optimality condition in fuzzy optimization problems and its application to classification of fuzzy data}


\author{Fangfang Shi\textsuperscript{1}, Guoju Ye\textsuperscript{1} Wei Liu\textsuperscript{1} \and Debdas Ghosh\textsuperscript{2,*}}
\affilOne{\textsuperscript{1} College of Sciences, Hohai University, China\\}
\affilTwo{\textsuperscript{2} Department of Mathematical Sciences, Indian Institute of Technology (BHU) Varanasi, Uttar Pradesh, 221 005, India}


\twocolumn[{

\maketitle

\begin{abstract}
The main objective of this paper is to derive the optimality conditions for one type of fuzzy optimization problems. At the beginning, we define a cone of descent direction for fuzzy optimization, and prove that its intersection with the cone of feasible directions at an optimal point is an empty set. Then, we present first-order optimality conditions for fuzzy optimization problems. Furthermore, we generalize the Gordan's theorem for fuzzy linear inequality systems and utilize it to deduce the Fritz-John optimality condition for the fuzzy optimization with inequality constraints. Finally, we apply the optimality conditions established in this paper to a binary classification problem for support vector machines with fuzzy data. In the meantime, numerical examples are described to demonstrate the primary findings proposed in the present paper.
\end{abstract}

\msinfo{20 June 2023}{****}{****}

\keywords{Fuzzy optimization, Gordan's theorem, Fritz-John optimality condition, Fuzzy data classification.}
}]

\setcounter{page}{1}
\corres
\volnum{ }
\issuenum{ }
\monthyear{ }
\pgfirst{ }
\pglast{ }
\doinum{ }
\articleType{}


\markboth{Shi, Ye, Liu and Ghosh}{FJ condition for fuzzy optimization}

\section{Introduction}
Optimization is a branch of applied mathematics, which aims to find the maximum or minimum value of a function under constraints. Optimization problems inevitably arise in many realistic fields including machine learning, management science, economics, physics, mechanics, etc. In practical problems, as the data are often generated by measurement and estimation, it is accompanied by uncertainties. These uncertain data govern the objective function and constraints conditions for the model optimization problem, if the coefficients in the representation involving the function are expressed as real numbers, the results may be meaningless as errors accumulate. Fuzzy set theory proposed by Zadeh \cite{Zadeh} is a powerful tool to solve uncertain problems when uncertainty is given by imprecision. It regards uncertain parameters as fuzzy numbers. This method has been applied in many fields of practical optimization, such as financial investment \cite{Yousefi2018}, cooperative game \cite{Liang2020}, support vector machine (SVM) \cite{Zhang2008} and so on.

In fuzzy optimization, the differentiability of fuzzy functions is an important idea in establishing optimality conditions. Panigrahi\cite{Panigrahi2008} used the Buckley-Feuring method to obtain the differentiability of multi-variable fuzzy mappings and deduced the KKT conditions for the constrained fuzzy minimization problem.  However, the main drawback of this method is that the fuzzy derivative is degenerated to the real function derivative, which effectively reduces the fuzziness. Wu \cite{Wu2007} defined the Hukuhara difference (H-difference) of two fuzzy numbers, and derived the weak and strong duality theorems of Wolfe primal-dual problem by defining the gradient of fuzzy function. Chalco-Cano et al. \cite{Chalco-Cano2013} considered the notion of strongly generalized differentiable fuzzy functions and obtained the KKT optimality conditions for fuzzy optimization problems (FOPs). Wu \cite{Wu2009} established the concepts of level-wise continuity and level-wise differentiability of fuzzy functions, and gave the KKT condition where the objective function is level continuous and differentiable at feasible solutions. The derivatives used in these methods all depend on the H-difference, but the H-difference does not necessarily exist for any pair of fuzzy numbers. Bede and Stefanini therefore introduced a new derivative in \cite{Bede} known as the generalized Hukuhara derivative (gH-derivative), that is more widespread than Hukuhara differentiability, level-wise differentiability and strongly generalized differentiability. Najariyan and Farahi \cite{Najariyan2015} investigated  fuzzy optimal control systems with fuzzy linear boundary conditions by utilizing the gH-differentiability of fuzzy functions. Chalco-Cano et al. \cite{Chalco-Cano2015} presented a Newton method for the search of non-dominant solutions of FOPs by using gH-differentiability. For more articles on fuzzy optimization, interested readers can refer to \cite{Osuna2016,Luhandjula2015,Stefanini2019,Zhang20222,Ghosh2023,Ghosh2022ORT,Ghosh2022CAM,Ghosh2022SC,Guo} and the reference therein.

It can be seen from the existing literature that the authors are committed to transforming fuzzy functions into interval-valued functions by using cut sets, and then using two upper and lower endpoint real functions to study the optimization conditions, which is a natural generalization of real functions. This technique of converting the FOP into a conventionally optimization model delivers a solution to the problem, but it neglects the analyse the overall solution. Nehi and Daryab \cite{Nehi2012} studied the optimality conditions of fuzzy optimization problems through the descending direction cone, which is a meaningful work, but unfortunately, the derivative used in the article is defined by the endpoint value function. In 2019, Ghosh et al. \cite{Ghosh2019} analyzed from the perspective of geometry solutions to constrained and unconstrained interval-valued problems, and obtained KKT optimality results.
Inspired by these literatures, the main objective of this paper is to study the Fritz-John and KKT type necessary optimality conditions for one type of FOPs with fuzzy constraints, and then apply the optimality conditions obtained in this paper to the SVM binary classification problems of fuzzy data sets.

The remainder of this article is to be organized as follows. The preliminaries are introduced in Section \ref{sec:2}. In Section \ref{sec:3}, we first put forward a definition of the cone descent direction of the fuzzy function, and use it to show the first-order optimality conditions of unconstrained FOP at the optimal solution. Then, Gordan's alternative theorem is extended to solve the system of fuzzy inequalities. Based on this theorem, the necessary conditions for Fritz-John and KKT types optimality of constrained FOPs are established, and numerical examples to verify the accuracy of the results are also available. It is worth noting that the means used in the proof are not to convert FOPs into traditional optimization models. In Section \ref{sec:4}, the findings of this paper are employed in binary classification problems with fuzzy data points. The study ends in Section \ref{sec:5}, where conclusions are drawn.
\section{Preliminaries}
\label{sec:2}
The set of all bounded and closed intervals in $\mathbb{R}$ is written as $\mathbb{R}_{\mathcal{I}}$, i.e.,
$$\mathbb{R}_{\mathcal{I}}=\{[\underline{\ell}, \overline{\ell}]: \underline{\ell}, \overline{\ell}\in\mathbb{R}~\text{and}~\underline{\ell}\leq \overline{\ell}\}.$$

A fuzzy set on $\mathbb{R}^{n}$ is a mapping $m: \mathbb{R}^{n}\rightarrow[0, 1]$, and its $\varrho$-cut set is $[m]^{\varrho}= \{x\in \mathbb{R}^{n}~: m(x)\geq \varrho\}$ for every $\varrho\in(0, 1]$, and the $0$-cut set is $[m]^{0}=\overline{\{x\in \mathbb{R}^{n}~: m(x)> 0\}}$, here $\overline{\mathbb{T}}$ represents the closure of the set $\mathbb{T}\subseteq\mathbb{R}^{n}$.

\begin{definition}\cite{Zadeh}
A fuzzy set $m$ on $\mathbb{R}$ is said to be a fuzzy number if
\begin{enumerate}[label=(\roman*)]
\item $m$ is normal, i.e., there exists $x^{*}\in \mathbb{R}$ such that $m(x^{*})=1$;
\item $m$ is upper semi-continuous;
\item $m(\vartheta x_{1}+(1-\vartheta) x_{2})\geq \min\{m(x_{1}), m(x_{2})\},~\forall~x_{1}, x_{2}\in \mathbb{R}, \vartheta\in [0, 1]$;
\item $[m]^{0}$ is compact.
\end{enumerate}
\end{definition}
The set of all fuzzy numbers is written as $\mathcal{F}_{\mathcal{C}}$. Let $m\in\mathcal{F}_{\mathcal{C}}$, the $\varrho$-cuts of $m$ are given $[m]^{\varrho}=[\underline{m}^{\varrho}, \overline{m}^{\varrho}]\in\mathbb{R}_{\mathcal{I}}$, where $\underline{m}^{\varrho}, \overline{m}^{\varrho}\in \mathbb{R}$ for each $\varrho\in[0, 1]$. Note that every $\ell\in\mathbb{R}$ can be regarded as $\tilde{\ell}\in\mathcal{F}_{\mathcal{C}}$, i.e., $[\tilde{\ell}]^{\varrho}=[\ell, \ell]$ for each $\varrho\in[0, 1]$.

For example, for a triangular fuzzy number $m=\langle\ell_{1}, \ell_{2}, \ell_{3}\rangle\in\mathcal{F}_{\mathcal{C}}$, $\ell_{1}, \ell_{2}, \ell_{3}\in\mathbb{R}$ and $\ell_{1}\leq \ell_{2}\leq \ell_{3}$, its $\varrho$-cut set is $[m]^{\varrho}=[\ell_{1}+(\ell_{2}-\ell_{1})\varrho, \ell_{3}-(\ell_{3}-\ell_{2})\varrho], \varrho\in[0, 1].$

Given $m, l\in\mathcal{F}_{\mathcal{C}}$, the distance between $m$ and $l$ is
$$\mathfrak{D}(m, l)=\sup_{\varrho\in[0, 1]} \max\{\mid\underline{m}^{\varrho}-\underline{l}^{\varrho}\mid, \mid\overline{m}^{\varrho}-\overline{l}^{\varrho}\mid\}.$$
For any $m, l\in\mathcal{F}_{\mathcal{C}}$ represented by $[\underline{m}^{\varrho}, \overline{m}^{\varrho}]$ and $[\underline{l}^{\varrho}, \overline{l}^{\varrho}]$, respectively, and for every $\vartheta\in\mathbb{R}$,
\begin{equation*}
(m+l)(x)=\sup_{x=x_{1}+x_{2}}\min\{m(x_{1}), l(x_{2})\},
\end{equation*}
\begin{equation*}(\vartheta m)(x)=\begin{cases}
  m(\frac{x}{\vartheta}),&\text{if }\vartheta\neq0,\\
  0,&\text{if }\vartheta=0,
  \end{cases}
\end{equation*}
respectively. For every $\varrho\in[0, 1]$,
\begin{equation*}
[m+l]^{\varrho}=[\underline{m}^{\varrho}+\underline{l}^{\varrho}, \overline{m}^{\varrho}+\overline{l}^{\varrho}],
\end{equation*}
\begin{equation*}
[\vartheta m]^{\varrho}=[\min\{\vartheta\underline{m}^{\varrho}, \vartheta\overline{m}^{\varrho}\}, \max\{\vartheta\underline{m}^{\varrho}, \vartheta\overline{m}^{\varrho}\}].
\end{equation*}

\begin{definition}\cite{Stefanini}
For any $m, l\in\mathcal{F}_{\mathcal{C}}$, their gH-difference $\gamma\in\mathcal{F}_{\mathcal{C}}$, if it exists, is given by
\begin{equation*}
 m\ominus_{gH}l=\gamma\Leftrightarrow\begin{cases}
 (i)~m=l+\gamma,\\
 \text{or}~(ii)~l=m+(-1)\gamma.
  \end{cases}
\end{equation*}
If $m\ominus_{gH}l$ exists, then for all $\varrho\in[0, 1]$,
\begin{equation*}
[m\ominus_{gH}l]^{\varrho}=[m]^{\varrho}\ominus_{gH}[l]^{\varrho}=[\min\{\underline{m}^{\varrho}-\underline{l}^{\varrho}, \overline{m}^{\varrho}-\overline{l}^{\varrho}\}, \max\{\underline{m}^{\varrho}-\underline{l}^{\varrho}, \overline{m}^{\varrho}-\overline{l}^{\varrho}\}].
\end{equation*}
\end{definition}

\begin{definition}\cite{Wu2007}
Let $m, l\in\mathcal{F}_{\mathcal{C}}$ be such that $[m]^{\varrho}=[\underline{m}^{\varrho}, \overline{m}^{\varrho}]$ and $[l]^{\varrho}=[\underline{l}^{\varrho}, \overline{l}^{\varrho}]$ for all $\varrho\in[0, 1]$. Then, we write
\begin{equation*}
m\preceqq l ~\text{iff}~ [m]^{\varrho}\preceqq_{LU} [l]^{\varrho} ~\text{for all}~\varrho\in[0, 1],
\end{equation*}
which is equivalent to writing $\underline{m}^{\varrho}\leq\underline{l}^{\varrho}$ and $\overline{m}^{\varrho}\leq\overline{l}^{\varrho}$ for all $\varrho\in[0, 1]$. We write
\begin{equation*}
m\preceq l ~\text{iff}~[m]^{\varrho}\preceq_{LU} [l]^{\varrho},
\end{equation*}
which is equivalent to $[m]^{\varrho}\preceqq_{LU} [l]^{\varrho}$ for all $\varrho\in[0, 1]$ and  $\underline{m}^{\varrho^{*}}<\underline{l}^{\varrho^{*}}$ or $\overline{m}^{\varrho^{*}}<\overline{l}^{\varrho^{*}}$ for some $\varrho^{*}\in[0, 1]$. We write
\begin{equation*}
m\prec l ~\text{iff}~ [m]^{\varrho}\prec_{LU} [l]^{\varrho} ~\text{for all}~\varrho\in[0, 1],
\end{equation*}
which is equivalent to saying that $\underline{m}^{\varrho}<\underline{l}^{\varrho}$ and $\overline{m}^{\varrho}<\overline{l}^{\varrho}$ for all $\varrho\in[0, 1]$.
\end{definition}
\begin{lemma}\label{lem2.3}
For any $m, l\in\mathcal{F}_{\mathcal{C}}$, $m\preceqq l$ iff $m\ominus_{gH}l\preceqq ~\tilde{0}$.
\end{lemma}
\begin{proof}
Let $m, l\in\mathcal{F}_{\mathcal{C}}$, for any $\varrho\in[0, 1]$, we have
\begin{equation*}
\begin{split}
~m\preceqq l
\Leftrightarrow&~[m]^{\varrho}\preceqq_{LU} [l]^{\varrho}\\
\Leftrightarrow&~\underline{m}^{\varrho}\leq\underline{l}^{\varrho}~ \text{and} ~\overline{m}^{\varrho}\leq\overline{l}^{\varrho} \\
\Leftrightarrow&~\underline{m}^{\varrho}-\underline{l}^{\varrho}\leq0~ \text{and} ~\overline{m}^{\varrho}-\overline{l}^{\varrho}\leq0\\
\Leftrightarrow&~[m]^{\varrho}\ominus_{gH}[l]^{\varrho}\preceqq_{LU}[0, 0]\\
\Leftrightarrow&~[m\ominus_{gH}l]^{\varrho}\preceqq_{LU}[0, 0]\\
\Leftrightarrow&~m\ominus_{gH}l\preceqq \tilde{0}.
\end{split}
\end{equation*}
The proof is completed.
\end{proof}
\begin{remark}\label{rem2.4}
From Lemma \ref{lem2.3}, it is obvious that $m\preceq l$ iff $m\ominus_{gH}l\preceq \tilde{0}$, $m\prec l$ iff $m\ominus_{gH}l\prec \tilde{0}$.
\end{remark}

\begin{definition}\cite{Nehi2012}
The $\mathcal{U}=(m_{1}, m_{2}, \ldots, m_{n})^{\top}$ is said to be an $n$-dimensional fuzzy vector
if $m_{1}, m_{2}, \ldots, m_{n}\in\mathcal{F}_{\mathcal{C}}$. All $n$-dimensional fuzzy vectors are recorded as $\mathcal{F}_{\mathcal{C}}^{n}$.\\
The $\varrho$-cut set of $\mathcal{U}=(m_{1}, m_{2}, \ldots, m_{n})^{\top}\in\mathcal{F}_{\mathcal{C}}^{n}$ is defined as
$$\mathcal{U}_{\varrho}=([m_{1}]^{\varrho}, [m_{2}]^{\varrho}, \ldots, [m_{n}]^{\varrho})^{\top}.$$
For any $\mathcal{U}=(m_{1}, m_{2}, \ldots, m_{n})^{\top}\in\mathcal{F}_{\mathcal{C}}^{n}$ and $\vartheta\in\mathbb{R}$, we have
$$\vartheta\mathcal{U}=(\vartheta m_{1}, \vartheta m_{2}, \ldots, \vartheta m_{n})^{\top},$$
and for any $\kappa=(\kappa_{1}, \kappa_{2}, \ldots, \kappa_{n})^{\top}\in\mathbb{R}^{n}$, the product $\kappa^{\top}\mathcal{U}$ is defined as $\kappa^{\top}\mathcal{U}=\sum_{j=1}^{n}\kappa_{j}m_{j}.$

For $\mathcal{U}=(m_{1}, m_{2}, \cdots, m_{n})^{\top}, \mathcal{V}=(l_{1}, l_{2}, \cdots, l_{n})^{\top} \in\mathcal{F}_{\mathcal{C}}^{n}$, then
\begin{equation*}
\begin{split}
&\mathcal{U} \preceqq \mathcal{V}~\text{iff}~ m_{j}\preceqq l_{j},~~\forall~j=1, 2, \cdots, n,\\
&\mathcal{U}\preceq \mathcal{V}~\text{iff}~ m_{j}\preceq l_{j},~~\forall~j=1, 2, \cdots, n,\\
&\mathcal{U}\prec \mathcal{V}~\text{iff}~  m_{j}\prec l_{j},~~\forall~j=1, 2, \cdots, n.
\end{split}
\end{equation*}
\end{definition}
Unless otherwise specified, $\mathbb{T}$ is represented as a nonempty subset of $\mathbb{R}^{n}$.

Let $\mathcal{H}: \mathbb{T}\rightarrow\mathcal{F}_{\mathcal{C}}$ be a fuzzy function on $\mathbb{T}$. For every $\varrho\in[0, 1]$, $\mathcal{H}$ can be represented as the set of interval-valued
functions $\mathcal{H}_{\varrho}: \mathbb{T}\rightarrow\mathbb{R}_{\mathcal{I}}$, as indicated below
$$\mathcal{H}_{\varrho}(x)=[\mathcal{H}(x)]^{\varrho}=[\underline{\mathcal{H}}^{\varrho}(x), \overline{\mathcal{H}}^{\varrho}(x)].$$

\begin{definition}\cite{Stefanini2009}\label{def22}
Let $x^{*}\in \mathbb{T}$. The $\mathcal{H}: \mathbb{T}\rightarrow \mathcal{F}_{\mathcal{C}}$ ia said to be continuous at $x^{*}$ if for any $\epsilon>0$, there is a $\delta>0$ that makes
$$\mathfrak{D}(\mathcal{H}(x), \mathcal{H}(x^{*}))<\epsilon~~\text{as}~ 0<|x-x^{*}|<\delta.$$
\end{definition}

\begin{theorem}\label{jubu}
Let the fuzzy function $\mathcal{H}: \mathbb{T}\rightarrow\mathcal{F}_{\mathcal{C}}$ be continuous at $x^{*}\in\mathbb{T}$ and $\mathcal{H}(x^{*})\prec\tilde{0}$. Then, there exists a $\delta>0$ such that $\mathcal{H}(x)\prec\tilde{0}$ whenever $0<|x-x^{*}|<\delta$.
\end{theorem}
\begin{proof}
It follows from the assumption that for any $\epsilon>0$, there is $\delta'>0$ as to make
$$\mathfrak{D}(\mathcal{H}(x), \mathcal{H}(x^{*}))<\epsilon ~~~\text{as}~~~ 0<|x-x^{*}|<\delta'.$$
Then,
\begin{equation*}
\mid\mathcal{\underline{H}}^{\varrho}(x)-\mathcal{\underline{H}}^{\varrho}(x^{*})\mid<\epsilon ~\text{and}~ \mid\mathcal{\overline{H}}^{\varrho}(x)-\mathcal{\overline{H}}^{\varrho}(x^{*})\mid<\epsilon ~\text{for all}~\varrho\in[0, 1].
\end{equation*}
So, $\mathcal{\underline{H}}^{\varrho}$ and $\mathcal{\overline{H}}^{\varrho}$ are continuous at $x^{*}$ for any $\varrho\in[0, 1]$.\\
From $\mathcal{H}(x^{*})\prec\tilde{0}$, we get
$$\underline{\mathcal{H}}^{\varrho}(x^{*})<0~\text{and}~\overline{\mathcal{H}}^{\varrho}(x^{*})<0~\text{for~all}~\varrho\in[0, 1].$$
Then, from the local sign preserving property of real continuous functions, there is a $\delta>0$ ($\delta\leq \delta'$) that make
\begin{equation*}
\mathcal{\underline{H}}^{\varrho}(x)<0 ~\text{and}~ \mathcal{\overline{H}}^{\varrho}(x)<0 ~\text{for each}~\varrho\in[0, 1]
\end{equation*}
as $0<|x-x^{*}|<\delta$. Then $\mathcal{H}_{\varrho}(x)\prec_{LU}[0, 0]$ for every $\varrho\in[0, 1]$, i.e., $\mathcal{H}(x)\prec\tilde{0}$.
\end{proof}

\begin{definition}\cite{Chalco-Cano}
We call the fuzzy function $\mathcal{H}$ convex on $\mathbb{T}$ if for each $x_{1}, x_{2}\in\mathbb{T}$ and all $\vartheta\in(0, 1)$,
\begin{equation*}
\mathcal{H}(\vartheta x_{1}+(1-\vartheta)x_{2})\preceq\vartheta \mathcal{H}(x_{1})+(1-\vartheta)\mathcal{H}(x_{2}).
\end{equation*}
\end{definition}

\begin{definition}\cite{Bede}
Let $\mathcal{H}: \mathbb{T}\subseteq \mathbb{R}\rightarrow\mathcal{F}_{\mathcal{C}}$ be a fuzzy function, $x^{*}\in \mathbb{T}$ and $d\in\mathbb{R}$ be such that $x^{*}+d\in \mathbb{T}$. Then, $\mathcal{H}$ is said to have gH-differentiable at $x^{*}$ if there exists $\mathcal{H}'(x^{*})\in\mathcal{F}_{\mathcal{C}}$ that make
\begin{equation*}\label{gH-d}
\mathcal{H}'(x^{*})=\lim_{d\rightarrow0}\frac{1}{d}(\mathcal{H}(x^{*}+d)\ominus_{gH}\mathcal{H}(x^{*})).
\end{equation*}
Moreover, the interval-valued function $\mathcal{H}_{\varrho}: \mathbb{T}\rightarrow\mathbb{R}_{\mathcal{I}}$ be gH-differentiable for every $\varrho\in[0, 1]$ and $[\mathcal{H}'(x)]^{\varrho}=\mathcal{H}_{\varrho}'(x)$.
\end{definition}

\begin{definition}\cite{Chalco-Cano2015}
Let $\mathcal{H}: \mathbb{T}\rightarrow\mathcal{F}_{\mathcal{C}}$ be a fuzzy function and $x^{*}=(x^{*}_{1}, x^{*}_{2}, \ldots, x^{*}_{n})^{\top}\in\mathbb{T}$ and let $\mathcal{P}_{i}(x_{i})=\mathcal{H}(x^{*}_{1}, \ldots, x^{*}_{i-1}, x_{i},$\\$ x^{*}_{i+1}, \ldots, x^{*}_{n})$. If $\mathcal{P}_{i}$ has gH-differentiable at $x_{i}^{*}$, then we
call $\mathcal{H}$ has the $i$th partial gH-differentiable at $x^{*}$ (written as $D_{i}\mathcal{H}(x^{*})$ and $D_{i}\mathcal{H}(x^{*})=\mathcal{P}_{i}'(x_{i}^{*})$).

The gradient $\nabla \mathcal{H}(x^{*})$ of $\mathcal{H}$ at $x^{*}$ is
\begin{equation*}
\nabla \mathcal{H}(x^{*})=\left(D_{1}\mathcal{H}(x^{*}), D_{2}\mathcal{H}(x^{*}), \ldots, D_{n}\mathcal{H}(x^{*})\right)^{\top}.
\end{equation*}
\end{definition}
\begin{definition}\cite{Stefanini2019}\label{def2.21}
Let $x^{*}\in\mathcal{X}$ and $\tau\in\mathbb{R}^{n}$. The fuzzy function $\mathcal{H}$ is called gH-directional differentiable at $x^{*}$ along the direction $\tau$, if there is a $\mathcal{H}_{\mathcal{D}}(x^{*})(\tau)\in\mathrm{E}$ that make
$$\lim_{\kappa\rightarrow0^{+}}\frac{1}{\kappa}(\mathcal{H}(x^{*}+\kappa \tau)\ominus_{gH}\mathcal{H}(x^{*}))=\mathcal{H}_{\mathcal{D}}(x^{*})(\tau).$$
\end{definition}
\begin{lemma}\cite{Stefanini2019}\label{lemm2.22}
Let fuzzy function $\mathcal{H}$ be gH-differentiable at $x^{*}\in\mathcal{X}$, then $\mathcal{H}_{\mathcal{D}}(x^{*})(\tau)=\tau^{\top}\nabla \mathcal{H}(x^{*})$ for any $\tau\in\mathbb{R}^{n}$.
\end{lemma}

\begin{theorem}
Let fuzzy function $\mathcal{H}: \mathbb{T}\rightarrow\mathcal{F}_{\mathcal{C}}$ be gH-differentiable on $\mathbb{T}$. If $\mathcal{H}$ is convex, then
\begin{equation*}
(x_{2}-x_{1})^{\top}\nabla \mathcal{H}(x_{1})\preceq \mathcal{H}(x_{2})\ominus_{gH}\mathcal{H}(x_{1}) ~\text{for all}~x_{1}, x_{2}\in\mathbb{T}.
\end{equation*}
\end{theorem}
\begin{proof}
Let $\mathcal{H}$ be convex on $\mathbb{T}$. For any $x_{1}, x_{2}\in\mathbb{T}$ and $\vartheta\in(0, 1)$,
\begin{equation*}
\begin{split}
\mathcal{H}(x_{1}+\vartheta (x_{2}-x_{1}))&=\mathcal{H}((1-\vartheta)x_{1}+\vartheta x_{2})\\
&\preceq(1-\vartheta)\mathcal{H}(x_{1})+\vartheta \mathcal{H}(x_{2}).
\end{split}
\end{equation*}
Therefore, by Lemma \ref{lem2.3},
\begin{equation*}
\begin{split}
&\mathcal{H}(x_{1}+\vartheta (x_{2}-x_{1}))\ominus_{gH}\mathcal{H}(x_{1})\preceq \vartheta (\mathcal{H}(x_{2})\ominus_{gH}\mathcal{H}(x_{1})),\\
\text{or,}~&\frac{1}{\vartheta}(\mathcal{H}(x_{1}+\vartheta (x_{2}-x_{1}))\ominus_{gH}\mathcal{H}(x_{1}))\preceq \mathcal{H}(x_{2})\ominus_{gH}\mathcal{H}(x_{1}),\\
\text{or,}~&\frac{1}{\vartheta}(\mathcal{H}(x_{1}+\vartheta (x_{2}-x_{1}))\ominus_{gH}\mathcal{H}(x_{1}))\preceq \mathcal{H}(x_{2})\ominus_{gH}\mathcal{H}(x_{1}).
\end{split}
\end{equation*}
Hence, as $\vartheta\rightarrow0^{+}$, we get
\begin{equation*}
(x_{2}-x_{1})^{\top}\nabla \mathcal{H}(x_{1})\preceq \mathcal{H}(x_{2})\ominus_{gH}\mathcal{H}(x_{1}).
\end{equation*}
\end{proof}

\section{Fritz-John optimality conditions}
\label{sec:3}
In this section, we will investigate the optimality conditions of FOPs at the optimal solution from the perspective of geometry.
\begin{theorem}\label{thm3.1}
Let $\mathcal{H}: \mathbb{T}\rightarrow\mathcal{F}_{\mathcal{C}}$ be gH-differentiable at $x^{*}\in\mathbb{T}$. If $\tau^{\top}\nabla \mathcal{H}(x^{*})\prec\tilde{0}$ for some $\tau\in\mathbb{R}^{n}$, then there is $\delta>0$ that makes for each $\kappa\in(0, \delta)$,
\begin{equation*}
\mathcal{H}(x^{*}+\kappa \tau)\prec \mathcal{H}(x^{*}).
\end{equation*}
\end{theorem}
\begin{proof}
Since $\mathcal{H}$ be gH-differentiable at $x^{*}$, then from Definition \ref{def2.21} and Lemma \ref{lemm2.22}, we have
\begin{equation*}
\lim_{\kappa\rightarrow0^{+}}\frac{1}{\kappa}(\mathcal{H}(x^{*}+\kappa \tau)\ominus_{gH}\mathcal{H}(x^{*}))=\tau^{\top}\nabla \mathcal{H}(x^{*}).
\end{equation*}
Since $\tau^{\top}\nabla \mathcal{H}(x^{*})\prec\tilde{0}$, we get
\begin{equation*}
\mathcal{H}(x^{*}+\kappa \tau)\ominus_{gH}\mathcal{H}(x^{*})\prec\tilde{0},
\end{equation*}
for each $\kappa\in(0, \delta)$ with some $\delta>0$.
From Remark \ref{rem2.4}, we get
\begin{equation*}
\mathcal{H}(x^{*}+\kappa \tau)\prec\mathcal{H}(x^{*}) ~\text{for any}~\kappa\in(0, \delta).
\end{equation*}
\end{proof}
\begin{remark}
The $\tau$ in Theorem \ref{thm3.1} represents a descent direction of $\mathcal{H}$ at $x^{*}$.
\end{remark}

\begin{definition}\label{def3.1}
Let $\mathcal{H}: \mathbb{T}\rightarrow\mathcal{F}_{\mathcal{C}}$ be gH-differentiable at $x^{*}\in\mathbb{T}$, we denote by $\hat{\mathcal{H}}(x^{*})$ the set of directions $\tau\in\mathbb{R}^{n}$ with $\tau^{\top}\nabla \mathcal{H}(x^{*})\prec\tilde{0}$, i.e.,
\begin{equation*}
\hat{\mathcal{H}}(x^{*})=\{\tau\in\mathbb{R}^{n}: \tau^{\top}\nabla \mathcal{H}(x^{*})\prec\tilde{0}\}.
\end{equation*}
As for any $\tau\in\mathbb{R}^{n}, \kappa \tau\in\hat{\mathcal{H}}(x^{*})$ for each $\kappa>0$, the set $\hat{\mathcal{H}}(x^{*})$ is said to be the cone of descent directions of $\mathcal{H}$ at $x^{*}$.
\end{definition}

\begin{definition}\label{def3.2}\cite{Bazaraa}
Given a $x^{*}\in\mathbb{T}$. The cone of feasible directions of $\mathbb{T}$ at $x^{*}$ is
\begin{equation*}
\hat{S}(x^{*})=\{\tau\in\mathbb{R}^{n}: x^{*}+\kappa \tau\in\mathbb{T}, \forall \kappa\in(0, \delta)~\text{for some}~\delta>0\}.
\end{equation*}
\end{definition}

\subsection{Unconstrained fuzzy optimization problems}
\qquad Consider the following unconstrained fuzzy optimization problem (UFOP, for short):
$$~~~\min_{x\in\mathbb{T}}\mathcal{H}(x),$$
where $\mathcal{H}: \mathbb{T}\rightarrow\mathcal{F}_{\mathcal{C}}$ be fuzzy function on $\mathbb{T}$.
\begin{definition}\cite{Chalco-Cano}
The point $x^{*}\in\mathbb{T}$ is said to be a (local) optimal solution of (UFOP) if there not exist $x\in \mathbb{T}(\in \mathcal{N}_{\delta}(x^{*})\cap\mathbb{T})$ such that $\mathcal{H}(x)\prec \mathcal{H}(x^{*})$, here $\mathcal{N}_{\delta}(x^{*})=\{x: 0<|x-x^{*}|<\delta, \delta>0\}$.
\end{definition}
\begin{theorem}\label{thm3.2}
If $\mathcal{H}$ be gH-differentiable at $x^{*}\in\mathbb{T}$ and $x^{*}$ is a local optimal solution of (UFOP), then $\hat{\mathcal{H}}(x^{*})\cap\hat{S}(x^{*})=\emptyset$.
\end{theorem}
\begin{proof}
On contrary, assume that $\hat{\mathcal{H}}(x^{*})\cap\hat{S}(x^{*})\not=\emptyset$ and $\tau\in\hat{\mathcal{H}}(x^{*})\cap\hat{S}(x^{*})$. Then, by Theorem \ref{thm3.1}, there is $\delta_{1}>0$ that makes
\begin{equation*}
\mathcal{H}(x^{*}+\kappa \tau)\prec \mathcal{H}(x^{*}),~~\kappa\in(0, \delta_{1}).
\end{equation*}
Also, by Definition \ref{def3.2}, there is $\delta_{2}>0$ that makes
$$x^{*}+\kappa \tau\in\mathbb{T},~~\kappa\in(0, \delta_{2}).$$
Defining $\delta = \min\{\delta_{1}, \delta_{2}\} > 0$, for all $\kappa\in(0, \delta)$, we obtain
$$x^{*}+\kappa \tau\in\mathbb{T}~ \text{and} ~\mathcal{H}(x^{*}+\kappa \tau)\prec \mathcal{H}(x^{*}).$$
This contradicts that $x^{*}$ be the local optimal solution. Thus, $\hat{\mathcal{H}}(x^{*})\cap\hat{S}(x^{*})=\emptyset$.
\end{proof}
\begin{cor}\label{cor3.2.1}
Let $\mathcal{H}$ be gH-differentiable at $x^{*}\in\mathbb{T}$. If $\hat{\mathcal{H}}(x^{*})\cap\hat{S}(x^{*})\neq\emptyset$, then $x^{*}$ is not a local optimal point of (UFOP).
\end{cor}

\begin{example}\label{exa1}
We consider the following UFOP:
$$\min_{(x_{1}, x_{2})\in\mathbb{T}} \mathcal{H}(x_{1}, x_{2}),$$
where $\mathcal{H}(x_{1}, x_{2})=\langle2, 3, 7\rangle(x_{1}-1.5)^{2}+\langle1, 2, 5\rangle x_{2}^{2}+1$ and $\mathbb{T}=\{(x_{2}, x_{2}): 1\leq x_{1}\leq2,~1\leq x_{2}\leq2\}$. Notice that for all $\varrho\in[0, 1]$,
$$\mathcal{H}_{\varrho}(x_{1}, x_{2})=[2+\varrho, 7-4\varrho](x_{1}-1.5)^{2}+[1+\varrho, 5-3\varrho]x_{2}^{2}+1.$$
The functions $\mathcal{H}_{\varrho}$ are depicted in Figures \ref{fig1} - \ref{fig3} for $\varrho=0, 0.3, 0.7$ and $ 1$ respectively. Evidently, $x^{*}=(1.5, 1)$ is an optimal point and
\begin{figure}
\subfigure[$\mathcal{H}_{\varrho}(x_{1}, x_{2})$ with $\varrho=0$.]
{
\begin{minipage}{6cm}
\centering
\includegraphics[scale=0.45]{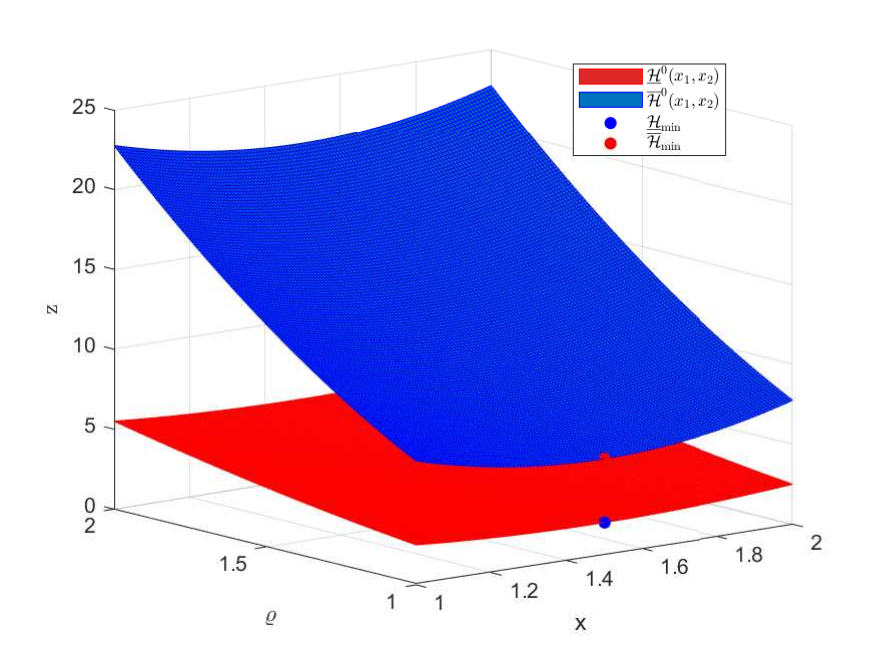}
\label{fig1}
\end{minipage}
}
\subfigure[$\mathcal{H}_{\varrho}(x_{1}, x_{2})$ with $\varrho=0.3$.]
{
\begin{minipage}{6cm}
\centering
\includegraphics[scale=0.45]{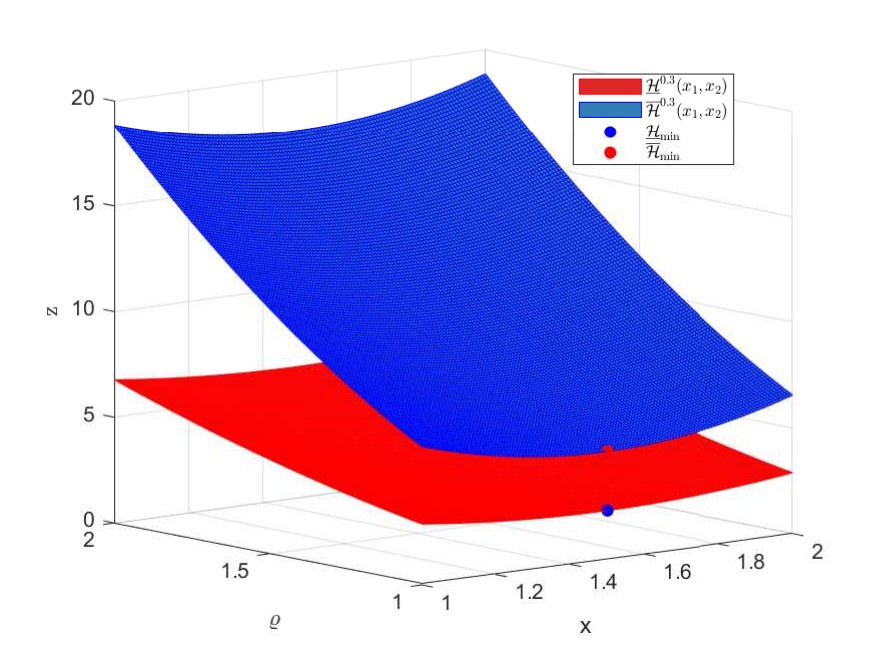}
\label{fig2}
\end{minipage}
}\newline
\subfigure[$\mathcal{H}_{\varrho}(x_{1}, x_{2})$ with $\varrho=0.7$.]
{
\begin{minipage}{6cm}
\centering
\includegraphics[scale=0.45]{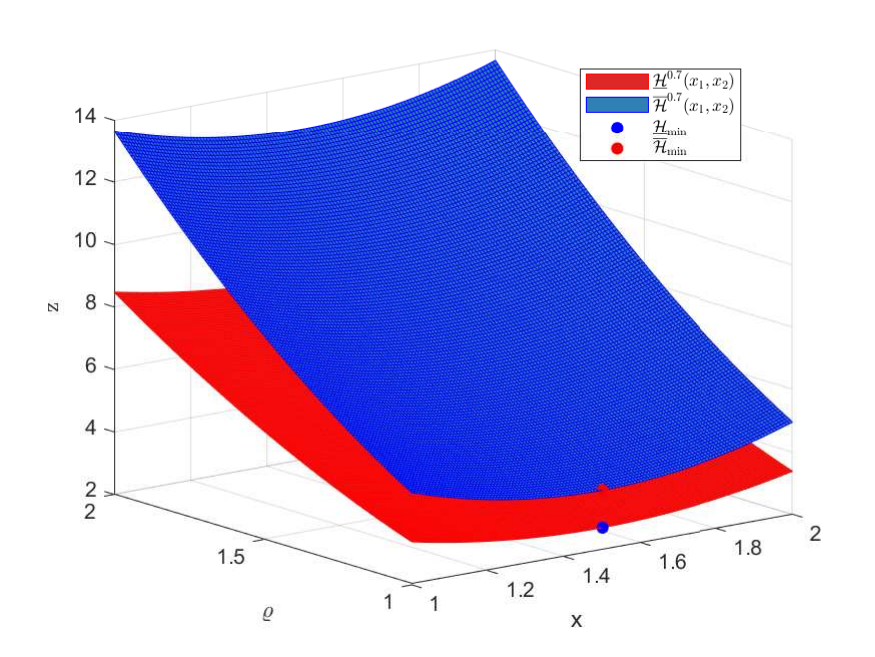}
\label{fig3}
\end{minipage}
}
\subfigure[$\mathcal{H}_{\varrho}(x_{1}, x_{2})$ with $\varrho=1$.]
{
\begin{minipage}{6cm}
\centering
\includegraphics[scale=0.45]{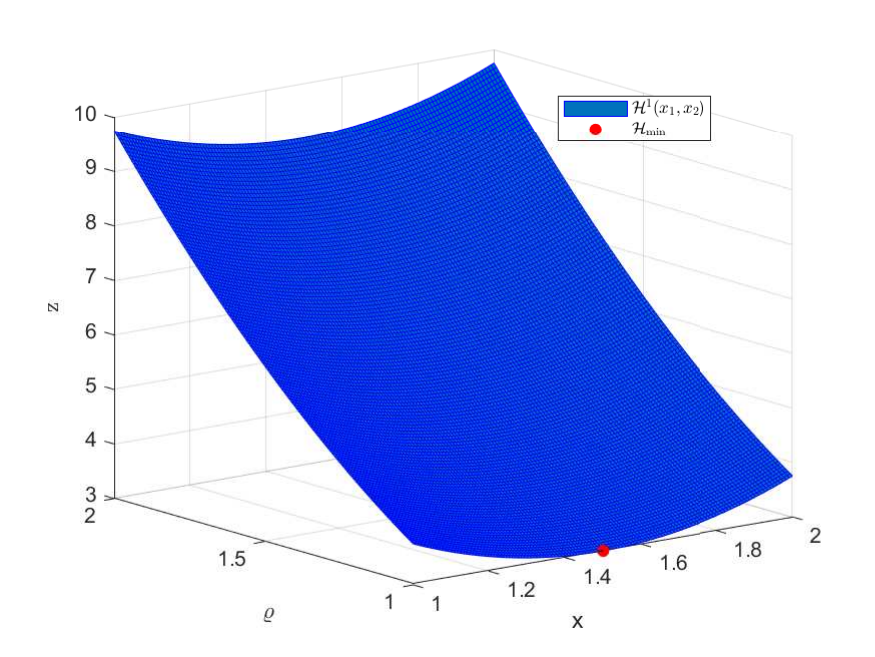}
\label{fig4}
\end{minipage}
}
\end{figure}
\begin{equation*}
\begin{split}
\hat{S}(x^{*})&=\{(\tau_{1}, \tau_{2})\not=(0, 0): (1.5+\kappa \tau_{1}, 1+\kappa \tau_{2})\in\mathbb{T}, \\
&~~~~~~\forall \kappa\in(0, \delta)~\text{for some}~\delta>0\}\\
&=\{(\tau_{1}, \tau_{2})\not=(0, 0): \tau_{2}\geq0\},
\end{split}
\end{equation*}
and
\begin{equation*}
\begin{split}
\hat{\mathcal{H}}(x^{*})&=\{(\tau_{1}, \tau_{2})\in\mathbb{R}^{2}: (\tau_{1}, \tau_{2})\nabla{\mathcal{H}}(x^{*})\prec\tilde{0}\}\\
&=\{(\tau_{1}, \tau_{2})\in\mathbb{R}^{2}: \tau_{1}D_{1}\mathcal{H}(x^{*})+\tau_{2}D_{2}\mathcal{H}(x^{*})\prec\tilde{0}\}\\
&=\{(\tau_{1}, \tau_{2})\in\mathbb{R}^{2}: \tau_{2}D_{2}\mathcal{H}_{\varrho}(x^{*})\prec_{LU}[0, 0]\\
&~~~~~~\text{for each}~\varrho\in[0, 1]\}\\
&=\{(\tau_{1}, \tau_{2})\in\mathbb{R}^{2}: \tau_{2}(2+2\varrho)<0~\text{and}~ \tau_{2}(10-6\varrho)<0\\
&~~~~~~\text{for each}~\varrho\in[0, 1]\}\\
&=\{(\tau_{1}, \tau_{2})\in\mathbb{R}^{2}: \tau_{2}<0\}.
\end{split}
\end{equation*}
Thus, we see that
$$\hat{S}(x^{*})\cap\hat{\mathcal{H}}(x^{*})=\emptyset.$$

Consider another point $\breve{x}=(1.5, 2)$.  By a similar way, we get
\begin{equation*}
\hat{S}(\breve{x})=\{(\tau_{1}, \tau_{2})\not=(0, 0): \tau_{2}\leq0\},
\end{equation*}
and
\begin{equation*}
\hat{\mathcal{H}}(\breve{x})=\{(\tau_{1}, \tau_{2})\in\mathbb{R}^{2}: \tau_{2}<0\}.
\end{equation*}
Since
\begin{equation*}
\hat{S}(\breve{x})\cap\hat{\mathcal{H}}(\breve{x})=\{(\tau_{1}, \tau_{2})\in\mathbb{R}^{2}: \tau_{2}<0\}\neq\emptyset.
\end{equation*}
Thus, from the Corollary \ref{cor3.2.1}, $\breve{x}$ must be not an optimal point of $\mathcal{H}$ on $\mathbb{T}$.
\end{example}

\begin{theorem}\label{thm3.4}(Fuzzy First Gordan's Theorem).
Let $\mathcal{U}=(m_{1}, m_{2}, \ldots, m_{n})^{\top}\in\mathcal{F}_{\mathcal{C}}^{n}$. Then, just a conclusion of the below holds:
\begin{enumerate}[label=(\roman*)]
\item there is $y=(y_{1}, y_{2}, \ldots, y_{n})^{\top}\in\mathbb{R}^{n}$ makes $y^{\top}\mathcal{U}\prec\tilde{0}$;
\item for each $\varrho\in[0, 1]$, there is $x\in\mathbb{R}, x>0$ makes $(0, 0, \ldots, 0)_{n}^{\top}\in x\mathcal{U}_{\varrho}$.
    \end{enumerate}
\end{theorem}
\begin{proof}
Suppose (i) holds. On contrary, assume that (ii) also holds.

As (i) holds, there exists $y_{0}\in\mathbb{R}^{n}$ such that
$$y_{0}^{\top}\mathcal{U}\prec\tilde{0}.$$
Then for any $x\in\mathbb{R}, x>0$,
$$x(y_{0}^{\top}\mathcal{U})\prec\tilde{0}\Longrightarrow y_{0}^{\top}(x\mathcal{U})\prec\tilde{0}.$$
That is,
\begin{equation}\label{(77)}
y_{0}^{\top}(x\mathcal{U}_{\varrho})\prec_{LU}[0, 0] ~\text{for all}~\varrho\in[0, 1].
\end{equation}
As (ii) is also true, then for all~$\varrho\in[0, 1]$~and exists~$x_{0}\in\mathbb{R}, x_{0}>0$,
$$(0, 0, \ldots, 0)_{n}^{\top}\in x_{0}\mathcal{U}_{\varrho}.$$
For all $y\in\mathbb{R}^{n}$, we have
\begin{equation}\label{(88)}
0\in y^{\top}(x_{0}\mathcal{U}_{\varrho}).
\end{equation}
Since \eqref{(77)} and \eqref{(88)} cannot hold simultaneously, a contradiction is obtained.

The other cases are proved below, suppose that (ii) holds, and (i) is false. On contrary, if (ii) is false, then for some $\varrho\in[0, 1]$ and all $x\in\mathbb{R}, x>0$,
$$(0, 0, \ldots, 0)_{n}^{\top}\notin x\mathcal{U}_{\varrho}\Longrightarrow (0, 0, \ldots, 0)_{n}^{\top}\notin \mathcal{U}_{\varrho}.$$
That is,
\begin{equation*}\label{(99)}
\begin{split}
&\exists~ j\in\{1, 2, \ldots, n\} ~\text{such that}~0\notin [m_{j}]^{\varrho},\\
\text{or,}~&\exists~ j\in\{1, 2, \ldots, n\} ~\text{such that}~[0, 0]\prec_{LU}[m_{j}]^{\varrho}~\text{or}~[m_{j}]^{\varrho}\prec_{LU}[0, 0].
\end{split}
\end{equation*}
Let
$$P=\{j: 0\in [m_{j}]^{\varrho}, j\in\{1, 2, \ldots, n\}\}$$
and
$$Q=\{j: 0\notin [m_{j}]^{\varrho}, j\in\{1, 2, \ldots, n\}\}.$$
Obviously, $Q\not=\emptyset$, $P\cup Q=\{1, 2, \ldots, n\}$ and $P\cap Q=\emptyset$.\\
Build a vector $y_{0}=(y^{0}_{1}, y^{0}_{2}, \ldots, y^{0}_{n})^{\top}$ through
\begin{equation*}
y^{0}_{j}=\begin{cases}
 0,&\text{if }j\in P,\\
 1,&\text{if }j\in Q~\text{and}~[m_{j}]^{\varrho}\prec_{LU}[0, 0],\\
 -1,&\text{if }j\in Q~\text{and}~[0, 0]\prec_{LU}[m_{j}]^{\varrho}.
  \end{cases}
\end{equation*}
Then
\begin{equation}\label{(1010)}
\begin{split}
&\sum_{j\in Q}y^{0}_{j}[m_{j}]^{\varrho}+\sum_{j\in P}y^{0}_{j}[m_{j}]^{\varrho}\prec_{LU}[0, 0],\\
\text{or,}~&\sum_{j=1}^{n}y^{0}_{j}[m_{j}]^{\varrho}\prec_{LU}[0, 0],\\
\text{or,}~&y_{0}^{\top}\mathcal{U}_{\varrho}\prec_{LU}[0, 0].
\end{split}
\end{equation}
However, as (i) is false, then there is no $y\in\mathbb{R}^{n}$ such that
\begin{equation*}
\begin{split}
&y^{\top}\mathcal{U}\prec\tilde{0},\\
\text{or,}~&y^{\top}\mathcal{U}_{\varrho}\prec_{LU}[0, 0] ~\text{for all}~\varrho\in[0, 1],
\end{split}
\end{equation*}
which is contradictory to \eqref{(1010)}. Therefore, we get the desired conclusion.
\end{proof}

Next we give the first-order optimality condition for the UFOP.
\begin{theorem}\label{thmm3.3}
Let $x^{*}$ is a local optimal solution of (UFOP). If $\mathcal{H}$ is gH-differentiable at $x^{*}$. Then, $(0, 0, \ldots, 0)_{n}^{\top}\in\nabla \mathcal{H}_{\varrho}(x^{*})$ for all $\varrho\in[0, 1]$.
\end{theorem}
\begin{proof}
From Definition \ref{def3.1} and Theorem \ref{thm3.1}, if $x^{*}$ be a local optimal solution, then $\hat{\mathcal{H}}(x^{*})=\emptyset$. Thus,
$$\tau^{\top}\nabla \mathcal{H}(x^{*})\prec\tilde{0}~\text{for no}~\tau\in \mathbb{R}^{n}.$$
By Theorem \ref{thm3.4}, for all $\varrho\in[0, 1]$, exists $x^{*}\in\mathbb{R}, x^{*}>0$,
$$(0, 0, \ldots, 0)_{n}^{\top}\in x^{*}\nabla\mathcal{H}_{\varrho}(x^{*})\Longrightarrow(0, 0, \ldots, 0)_{n}^{\top}\in \nabla\mathcal{H}_{\varrho}(x^{*}).$$
\end{proof}

\subsection{Fuzzy optimization problem with inequality constraints}
Consider the following Fuzzy optimization problem (FOP, for short):
\begin{equation*}\label{(66)}
\begin{split}
\min ~~&\mathcal{H}(x),\\
~~~~~~~~~~~~~~\text{subject to}~~&\mathcal{Y}_{j}(x)\preceq\tilde{0}~~\text{for}~j=1, 2, \ldots, s,\\
&x\in\mathbb{T},
\end{split}
\end{equation*}
where $\mathcal{H}: \mathbb{T}\rightarrow\mathcal{F}_{\mathcal{C}}$ and $\mathcal{Y}_{j}: \mathbb{T}\rightarrow\mathcal{F}_{\mathcal{C}}$, $j=1, 2, \ldots, s$. The feasible set of (FOP) is
$$\mathcal{O}=\{x\in\mathbb{T}: \mathcal{Y}_{j}(x)\preceq\tilde{0}~\text{for every}~j=1, 2, \ldots, s\}.$$

\begin{lemma}\label{lem3.1}
Let $\mathcal{Y}_{j}: \mathbb{T}\rightarrow\mathcal{F}_{\mathcal{C}}$, $j=1, 2, \ldots, s$, be fuzzy functions on the open set $\mathbb{T}$ and $\mathbb{S}=\{x\in\mathbb{T}: \mathcal{Y}_{j}(x)\preceq\tilde{0} ~\text{for}~ j = 1, 2, \ldots, s\}$. Let $x^{*}\in\mathbb{S}$ and $\Lambda(x^{*})=\{j: \mathcal{Y}_{j}(x^{*})=\tilde{0}\}$. Assuming that $\mathcal{Y}_{j}$ be gH-differentiable at $x^{*}$ for $j\in \Lambda(x^{*})$ and continuous at $x^{*}$ for $j\not\in \Lambda(x^{*})$, define
\begin{equation*}
\hat{\mathcal{Y}}(x^{*})=\{\tau\in\mathbb{R}^{n}: \tau^{\top}\nabla \mathcal{Y}_{j}(x^{*})\prec\tilde{0},~j\in \Lambda(x^{*})\}.
\end{equation*}
Then,
\begin{equation*}
\hat{\mathcal{Y}}(x^{*})\subseteq\hat{S}(x^{*}),
\end{equation*}
here $\hat{S}(x^{*})=\{\tau\in\mathbb{R}^{n}: x^{*}+\kappa \tau\in\mathbb{S}, \forall \kappa\in(0, \delta),~\delta>0\}.$
\end{lemma}
\begin{proof}
Let $\tau\in\hat{\mathcal{Y}}(x^{*})$. As $x^{*}\in\mathbb{T}$ and exists some $\delta_{0}>0$ makes
\begin{equation}\label{(33)}
x^{*}+\kappa\tau\in\mathbb{T}~~\text{for}~\kappa\in(0, \delta_{0}).
\end{equation}
Since $\mathcal{Y}_{j}$ be continuous at $x^{*}$ and $\mathcal{Y}_{j}(x^{*})\prec\tilde{0}$ for every $j\not\in \Lambda(x^{*})$, then Theorem \ref{jubu} shows that there have $\delta_{i}>0$ leading to
\begin{equation}\label{(44)}
\mathcal{Y}_{j}(x^{*}+\kappa\tau)\prec\tilde{0}~~\text{for}~\varrho\in(0, \delta_{i}),~j\not\in \Lambda(x^{*}).
\end{equation}
Also, as $\tau\in\hat{\mathcal{Y}}(x^{*})$, from Theorem \ref{thm3.1}, for every $j\in \Lambda(x^{*})$, there have $\delta_{j}>0$ makes
\begin{equation}\label{(55)}
\mathcal{Y}_{j}(x^{*}+\kappa\tau)\prec \mathcal{Y}_{j}(x^{*})=\tilde{0}~~\text{for each}~\kappa\in(0, \delta_{j}).
\end{equation}
Let $\delta=\min\{\delta_{0}, \delta_{1}, \ldots, \delta_{s}\}$. From \eqref{(33)}-\eqref{(55)}, $x^{*}+\kappa\tau\in\mathbb{S}$ for every $\kappa\in(0, \delta)$. That is, $\tau\in\hat{S}(x^{*})$. Hence, the proof is completed.
\end{proof}
With the help of Lemma \ref{lem3.1}, we can characterize the local optimal solution of (FOP).

\begin{theorem}\label{thm3..10}
For $x^{*}\in\mathcal{O}$, define $\Lambda(x^{*})=\{i: \mathcal{Y}_{j}(x^{*})=\tilde{0}\}$. Suppose that $\mathcal{H}$ and $\mathcal{Y}_{j}$ ($j\in \Lambda(x^{*})$) be gH-differentiable and $\mathcal{Y}_{j}$, $j\not\in \Lambda(x^{*})$, be continuous at $x^{*}$. If $x^{*}$ is a local optimal solution of (FOP), then
$$\hat{\mathcal{H}}(x^{*})\cap\hat{\mathcal{Y}}(x^{*})=\emptyset,$$
where $\hat{\mathcal{H}}(x^{*})=\{\tau\in\mathbb{R}^{n}: \tau^{\top}\nabla \mathcal{H}(x^{*})\prec\tilde{0}\}$ and $\hat{\mathcal{Y}}(x^{*})=\{\tau\in\mathbb{R}^{n}: \tau^{\top}\nabla \mathcal{Y}_{j}(x^{*})\prec\tilde{0} ~\text{for}~j\in \Lambda(x^{*})\}$.
\end{theorem}
\begin{proof}
Using Theorem \ref{thm3.2} and Lemma \ref{lem3.1}, we can get
\begin{equation*}
\begin{split}
x^{*}~ \text{is a local optimal solution}&\Longrightarrow \hat{\mathcal{H}}(x^{*})\cap\hat{S}(x^{*})=\emptyset\\
&\Longrightarrow\hat{\mathcal{H}}(x^{*})\cap\hat{\mathcal{Y}}(x^{*})=\emptyset.
\end{split}
\end{equation*}
\end{proof}
\begin{definition}\cite{Nehi2012}
The $\mathcal{M}=(m_{ij})_{s\times n}$ is said to be a fuzzy matrix if $m_{ij}\in\mathcal{F}_{\mathcal{C}}$ for every $i\in\{1, \ldots, s\}, j\in\{1, \ldots, n\}$, denoted by
$$
\mathcal{M}=\begin{pmatrix}
m_{11}&\ldots & m_{1n}\\
\vdots &\ddots & \vdots\\
m_{s1}& \ldots & m_{sn}
\end{pmatrix}_{s\times n}.
$$
The $\varrho$-cut set of $\mathcal{M}$ is defined as
$$
\mathcal{M}_{\varrho}=\begin{pmatrix}
[m_{11}]^{\varrho}&\ldots & [m_{1n}]^{\varrho}\\
\vdots &\ddots & \vdots\\
[m_{s1}]^{\varrho}& \ldots & [m_{sn}]^{\varrho}
\end{pmatrix}_{s\times n}.
$$
\end{definition}
\begin{theorem}\label{thm3.6}(Fuzzy Second Gordan's Theorem).
Consider a fuzzy matrix $\mathcal{M}=(m_{ij})_{s\times n}$. Then, just a conclusion of the below holds:
\begin{enumerate}[label=(\roman*)]
 \item there exists $y=(y_{1}, y_{2}, \ldots, y_{s})^{\top}\in\mathbb{R}^{s}$ makes $\mathcal{M}^{\top}y\prec(\tilde{0}, \tilde{0}, \ldots, \tilde{0})_{n}^{\top}$;
\item~ for all $\varrho\in[0, 1]$, there exists nonzero $x=(x_{1}, x_{2}, \ldots, x_{n})^{\top}\in\mathbb{R}^{n}$ with all $x_{i}\geq0$ such that $(0, 0, \ldots, 0)^{\top}_{s}\in \mathcal{M}_{\varrho}x$.
    \end{enumerate}
\end{theorem}
\begin{proof}
Suppose (i) holds. On contrary, assume that (ii) also holds.

As (i) holds, there exists $y_{0}=(y_{1}^{0}, y_{2}^{0}, \ldots, y_{s}^{0})^{\top}\in\mathbb{R}^{s}$ such that $\mathcal{M}^{\top}y_{0}\prec(\tilde{0}, \tilde{0}, \ldots, \tilde{0})_{n}^{\top}$. Then, for all $\varrho\in[0, 1]$,
$$\mathcal{M}_{\varrho}^{\top}y_{0}\prec_{LU}([0, 0], [0, 0], \ldots, [0, 0])_{n}^{\top}.$$
For all nonzero $x=(x_{1}, x_{2}, \ldots, x_{n})^{\top}\in\mathbb{R}^{n}, x_{i}\geq0$, we have
\begin{equation}\label{(777)}
x^{\top}(\mathcal{M}_{\varrho}^{\top}y_{0})\prec_{LU}[0, 0] \Longrightarrow(x\mathcal{M}_{\varrho})^{\top}y_{0}\prec_{LU}[0, 0].
\end{equation}
As (ii) is also true, there exists a nonzero $x_{0}=(x^{0}_{1}, x^{0}_{2}, \ldots, x^{0}_{n})^{\top}$\\
$\in\mathbb{R}^{n},~x_{j}^{0}\geq0$ such that
\begin{equation*}\label{(888)}
(0, 0, \ldots, 0)^{\top}_{s}\in \mathcal{M}_{\varrho}x_{0}.
\end{equation*}
For any given $\varrho\in[0, 1]$, let $w=\mathcal{M}_{\varrho}x_{0}=(w_{1}, w_{2}, \ldots, w_{s})^{\top}$. Then, $0\in w_{j}~ \text{for all} ~j=1, 2, \ldots, s$, and
\begin{equation*}
(\mathcal{M}_{\varrho}x_{0})^{\top}y_{0}=\sum_{j=1}^{s}y_{j}^{0}w_{j}.
\end{equation*}
Thus,
\begin{equation}\label{(aaa)}
0\in(\mathcal{M}_{\varrho}x_{0})^{\top}y_{0}.
\end{equation}
Since \eqref{(777)} and \eqref{(aaa)} cannot hold simultaneously, a contradiction is obtained.

The other cases are proved below, suppose that (ii) holds, and (i) is false. From (i) is false knows that there no $y\in \mathbb{R}^{s}$ makes $\mathcal{M}^{\top}y\prec(\tilde{0}, \tilde{0}, \ldots, \tilde{0})_{n}^{\top}$. That is,
\begin{equation}\label{(8888)}
\mathcal{M}_{\varrho}^{\top}y\prec_{LU}([0, 0], [0, 0], \ldots, [0, 0])_{n}^{\top} ~\text{for all}~\varrho\in[0, 1].
\end{equation}
On contrary, assume that (ii) is false. Then, for any nonzero $x=(x_{1}, x_{2}, \ldots, x_{n})^{\top}\in\mathbb{R}^{n}$ with every~$x_{j}\geq0$ and for some $\varrho\in[0, 1]$,
\begin{equation*}\label{(999)}
\begin{split}
&(0, 0, \ldots, 0)_{s}^{\top}\notin \mathcal{M}_{\varrho}x,~\\
\text{or,}~&\exists ~j\in\{1, \ldots, s\} ~\text{so that}~0\notin w_{j},\\
\text{or,}~&\exists ~j\in\{1, \ldots, s\} ~\text{so that}~[0, 0]\prec_{LU} w_{j}~\text{or}~w_{j}\prec_{LU}[0, 0],
\end{split}
\end{equation*}
where $\mathcal{M}_{\varrho}x=(w_{1}, w_{2}, \ldots, w_{s})^{\top}.$\\
Let
$$P=\{j: 0\in w_{j}, j\in\{1, \ldots, s\}\}$$
and
$$Q=\{j: 0\notin w_{j}, j\in\{1, \ldots, s\}\}.$$
Obviously, $Q\not=\emptyset$, $P\cup Q=\{1, \ldots, s\}$ and $P\cap Q=\emptyset$.\\
Build a vector $y_{0}=(y_{1}^{0}, y_{2}^{0}, \ldots, y_{s}^{0})^{\top}$ through
\begin{equation*}
y_{j}^{0}=\begin{cases}
 0,&\text{if }j\in P,\\
 1,&\text{if }j\in Q~\text{and}~w_{j}\prec_{LU}[0, 0],\\
 -1,&\text{if }j\in Q~\text{and}~[0, 0]\prec_{LU}w_{j}.
  \end{cases}
\end{equation*}
Then
\begin{equation*}
\begin{split}
&\sum_{j\in Q}y_{j}^{0}w_{j}+\sum_{j\in P}y_{j}^{0}w_{j}\prec_{LU}[0, 0],\\
\text{or,}~&y_{0}^{\top}\mathcal{M}_{\varrho}\prec_{LU}[0, 0].\\
\end{split}
\end{equation*}
So, for any nonzero $x=(x_{1}, x_{2}, \ldots, x_{n})^{\top}\in\mathbb{R}^{n}$~with every~$x_{j}\geq0$, we get
\begin{equation}\label{(10100)}
\begin{split}
&~y_{0}^{\top}(\mathcal{M}_{\varrho}x)\prec_{LU}[0, 0], \\
\text{or,}~&x^{\top}(\mathcal{M}_{\varrho}y_{0})\prec_{LU}[0, 0].
\end{split}
\end{equation}
The inequality \eqref{(10100)} can be true only when $\mathcal{M}_{\varrho}^{\top}y_{0}\prec_{LU}([0, 0],$\\
 $[0, 0], \ldots, [0, 0])_{n}^{\top}$. As \eqref{(8888)} and \eqref{(10100)} are contradictory, we get the desired conclusion.
\end{proof}

\begin{theorem}\label{thm3.7}(Fuzzy Fritz-John necessary condition).
Suppose $\mathcal{H}$ and
$\mathcal{Y}_{j}: \mathbb{T}\rightarrow\mathcal{F}_{\mathcal{C}}$ for $j=1, 2, \ldots, s$ be fuzzy functions.
Let $x^{*}$ is a local optimal point of (FOP) and define $\Lambda(x^{*})=\{j: \mathcal{Y}_{j}(x^{*})=\tilde{0}\}$. If $\mathcal{H}$ and $\mathcal{Y}_{j}$ ($j\in \Lambda(x^{*})$) are gH-differentiable and $\mathcal{Y}_{j}$ ($j\notin \Lambda(x^{*})$) be continuous at $x^{*}$, then there have $\kappa_{0}, \kappa_{j}\in\mathbb{R}$, $j\in \Lambda(x^{*})$ makes
\begin{equation*}
\begin{cases}
&(0, 0, \ldots, 0)_{n}^{\top}\in (\kappa_{0}\nabla \mathcal{H}_{\varrho}(x^{*})+\sum_{j\in \Lambda(x^{*})}\kappa_{j}\nabla {\mathcal{Y}_{j}}_{\varrho}(x^{*}))\\
&~~~~\text{for all}~\varrho\in[0, 1],\\
&\kappa_{0}\geq0, \kappa_{j}\geq0,~j\in \Lambda(x^{*}),\\
&(\kappa_{0}, \kappa_{\Lambda})\not=(0, 0^{\mid\Lambda(x^{*})\mid}),
\end{cases}
\end{equation*}
where $\kappa_{\Lambda}$ denotes a vector and its components are $\kappa_{j}$, $j\in \Lambda(x^{*})$.

Also, if $\mathcal{Y}_{j}$ ($j\notin \Lambda(x^{*})$) be also gH-differentiable at $x^{*}$, then there are $\kappa_{0}, \kappa_{1}, \ldots, \kappa_{s}$ that makes
\begin{equation*}
\begin{cases}
&(0, 0, \ldots, 0)_{n}^{\top}\in (\kappa_{0}\nabla \mathcal{H}_{\varrho}(x^{*})+\sum_{j=1}^{s}\kappa_{j}\nabla {\mathcal{Y}_{j}}_{\varrho}(x^{*}))\\
&~~~~~\text{for all}~\varrho\in[0, 1],\\
&\kappa_{i}\mathcal{Y}_{j}(x^{*})=\tilde{0}, j=1, 2, \ldots, s,\\
&\kappa_{0}\geq0, \kappa_{j}\geq0, j=1, 2, \ldots, s,\\
&(\kappa_{0}, \kappa)\not=(0, 0^{s}).
\end{cases}
\end{equation*}
\end{theorem}
\begin{proof}
Since $x^{*}$ be local optimal point of (FOP), from Theorem \ref{thm3..10}, we know that  $\hat{\mathcal{H}}(x^{*})\cap\hat{\mathcal{Y}}(x^{*})=\emptyset$. Then, there no $\tau\in\mathbb{R}^{n}$ that makes
\begin{equation}\label{2020}
\tau^{\top}\nabla \mathcal{H}(x^{*})\prec\tilde{0} ~\text{and}~\tau^{\top}\nabla \mathcal{Y}_{j}(x^{*})\prec\tilde{0}, \forall j\in \Lambda(x^{*}).
\end{equation}
Set $\mathcal{M}$ be a matrix with column $\nabla \mathcal{H}(x^{*})$ and $\nabla \mathcal{Y}_{j}(x^{*}), j\in \Lambda(x^{*})$, i.e.,
$$\mathcal{M}=[\nabla \mathcal{H}(x^{*}), ~[\nabla \mathcal{Y}_{j}(x^{*})]_{j\in \Lambda(x^{*})}]_{n\times (1+\mid\Lambda(x^{*})\mid)}.$$
It can be obtained from \eqref{2020} that
\begin{equation}\label{2121}
\mathcal{M}^{\top}\tau\prec (\tilde{0}, \tilde{0}, \ldots, \tilde{0})^{\top}_{1+\mid\Lambda(x^{*})\mid}~\text{for no}~\tau\in \mathbb{R}^{n}.
\end{equation}
So, we know from Theorem \ref{thm3.6}, there is a nonzero $\eta=(\eta_{j})_{\mid\Lambda(x^{*})+1\mid\times1}\in\mathbb{R}^{\mid\Lambda(x^{*})+1\mid}, \eta_{j}\geq0$ makes $(0, 0, \ldots, 0)_{n}^{\top}\in \mathcal{M}_{\varrho}\eta$ for all $\varrho\in[0, 1]$. Let
\begin{equation}\label{2222}
\eta=[\kappa_{0}, \kappa_{j}]^{\top}_{j\in \Lambda(x^{*})}.
\end{equation}
Substituting \eqref{2222} in $(0, 0, \ldots, 0)_{n}^{\top}\in \mathcal{M}_{\varrho}\eta$, we get
\begin{equation*}
\begin{cases}
&(0, 0, \ldots, 0)_{n}^{\top}\in (\kappa_{0}\nabla \mathcal{H}_{\varrho}(x^{*})+\sum_{j\in \Lambda(x^{*})}\kappa_{j}\nabla {\mathcal{Y}_{j}}_{\varrho}(x^{*}))\\
&~~~~~\text{for all}~\varrho\in[0, 1],\\
&\kappa_{0}\geq0, \kappa_{j}\geq0,~j\in \Lambda(x^{*}),\\
&(\kappa_{0}, \kappa_{\Lambda})\not=(0, 0^{\mid\Lambda(x^{*})\mid}).
\end{cases}
\end{equation*}
Thus the first part of Theorem \ref{thm3.7} is proved.

Since $\mathcal{Y}_{j}(x^{*})=\tilde{0}, j\in \Lambda(x^{*})$, we get $\kappa_{j} \mathcal{Y}_{j}(x^{*})=\tilde{0}$. If $\mathcal{Y}_{j}$ ($j\notin \Lambda(x^{*})$) be gH-differentiable at $x^{*}$, let $\kappa_{j}=0$ $(j\notin \Lambda(x^{*})$), we get the second part of Theorem \ref{thm3.7}.
\end{proof}

\begin{example}\label{exa2}
Consider the FOP:
\begin{equation*}
\begin{split}
\min~~&\mathcal{H}(x)=\langle-2, -1, 1\rangle x^{2}+\langle-8, -4, 3\rangle x+\langle1, 2, 4\rangle,\\
\text{subject to}~~&\mathcal{Y}_{1}(x)=\langle-4, 5, 7\rangle x\ominus_{gH}\langle-8, 10, 14\rangle\preceq\tilde{0},\\
&\mathcal{Y}_{2}(x)=\langle-3, -2, 0\rangle x\ominus_{gH}\langle2, 3, 6\rangle \preceq\tilde{0}.
\end{split}
\end{equation*}
Obviously, $\mathcal{H}, \mathcal{Y}_{1}$ and $\mathcal{Y}_{2}$ are gH-differentiable on $(0, +\infty)$. The image of the objective function $\mathcal{H}$ is shown in Figure \ref{fff}. At the feasible point $x^{*}=2$, we have
$$\mathcal{Y}_{1}(2)=\tilde{0}~\text{and}~\mathcal{Y}_{2}(2)=\langle-8, -7, -6\rangle.$$
Hence, $\Lambda(x^{*})=\{1\}$, and we get
\begin{equation*}
\begin{split}
& \nabla \mathcal{H}_{\varrho}(2)=[-16+8\varrho, 7-15\varrho],\\
& \nabla {\mathcal{Y}_{1}}_{\varrho}(2)=[-4+9\varrho, 7-2\varrho].
\end{split}
\end{equation*}
Taking $\kappa_{0}=5, \kappa_{1}=8$ and $\kappa_{2}=0$, the conclusion of Theorem \ref{thm3.7} is true.

\begin{figure}[h]
  \vspace{0cm}
  \centering
  \includegraphics[width=8cm]{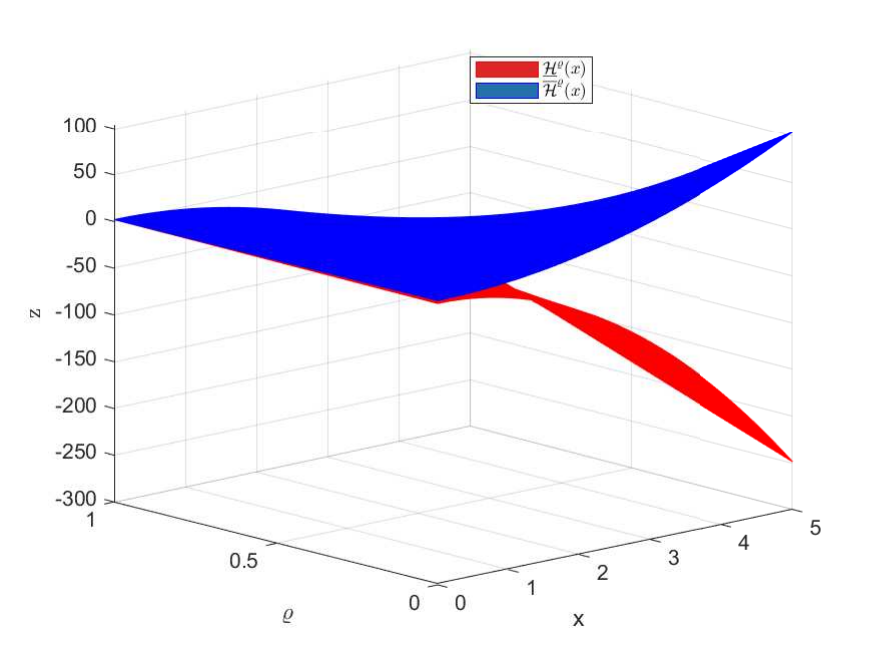}
  \vspace{0cm}
  \centering
  \caption{The objective function $\mathcal{H}_{\varrho}(x)$ in Example \ref{exa2}.}\label{fff}
\end{figure}
\end{example}

\begin{definition}
The set of $n$ fuzzy vectors $\{\mathcal{U}_{1}, \mathcal{U}_{2}, \ldots, \mathcal{U}_{n}\}$ is said to be linearly independent if for $\vartheta_{1}, \vartheta_{2}, \ldots, \vartheta_{n}\in\mathbb{R}$, $(0, 0, \ldots, 0)_{n}^{\top}\in \vartheta_{1}{\mathcal{U}_{1}}_{\varrho}+\vartheta_{2}{\mathcal{U}_{2}}_{\varrho}+\ldots+\vartheta_{n}{\mathcal{U}_{n}}_{\varrho}~\text{holds for all}~ \varrho\in[0, 1]$ $~\text{if and only if}~\vartheta_{1}=\vartheta_{2}=\ldots=\vartheta_{n}=0.$
\end{definition}

\begin{theorem} \label{thm3.8}(Fuzzy KKT necessary condition).
Let $\mathcal{H}, \mathcal{Y}_{j}: \mathbb{T}\rightarrow\mathcal{F}_{\mathcal{C}}, j=1, 2, \ldots, s$, be fuzzy functions on the open set $\mathbb{T}$ and $x^{*}$ is a local optimal solution and define $\Lambda(x^{*})=\{j: \mathcal{Y}_{j}(x^{*})=\tilde{0}\}$. Assume that~$\mathcal{H}$ and $\mathcal{Y}_{j}$ ($j\in \Lambda(x^{*})$) are gH-differentiable and~$\mathcal{Y}_{j}$ ($j\notin \Lambda(x^{*})$) be continuous at $x^{*}$. If~the set of fuzzy vectors $\{\nabla {\mathcal{Y}_{j}}(x^{*}): j\in \Lambda(x^{*})\}$ is linearly independent,
then there have $\kappa_{j}\in\mathbb{R}$, $j\in \Lambda(x^{*})$ makes
\begin{equation*}
\begin{cases}
&(0, 0, \ldots, 0)_{n}^{\top}\in (\nabla \mathcal{H}_{\varrho}(x^{*})+\sum_{j\in \Lambda(x^{*})}\kappa_{i}\nabla {\mathcal{Y}_{j}}_{\varrho}(x^{*}))\\
&~~~~\text{for all}~\varrho\in[0, 1],\\
&\kappa_{j}\geq0 ~\text{for}~j\in \Lambda(x^{*}).
\end{cases}
\end{equation*}
If $\mathcal{Y}_{j}$ is also gH-differentiable at $x^{*}$ for $j\notin I(x^{*})$, then there are $\kappa_{1}, \kappa_{2}, \ldots, \kappa_{s}$ so that
\begin{equation*}
\begin{cases}
&(0, 0, \ldots, 0)_{n}^{\top}\in (\nabla \mathcal{H}_{\varrho}(x^{*})+\sum_{j=1}^{s}\kappa_{j}\nabla {\mathcal{Y}_{j}}_{\varrho}(x^{*}))\\
&~~~~\text{for all}~\varrho\in[0, 1],\\
&\kappa_{j}\mathcal{Y}_{j}(x^{*})=\tilde{0}, j=1, 2, \ldots, s,\\
&\kappa_{j}\geq0 ~\text{for}~j=1, 2, \ldots, s.\\
\end{cases}
\end{equation*}
\end{theorem}
\begin{proof}
By Theorem \ref{thm3.7}, there have $\kappa_{0}, \kappa'_{j}\in\mathbb{R}$ for all $j\in \Lambda(x^{*})$ not all zeros such that
\begin{equation*}
\begin{cases}
&(0, 0, \ldots, 0)_{n}^{\top}\in (\kappa_{0}\nabla \mathcal{H}_{\varrho}(x^{*})+\sum_{j\in \Lambda(x^{*})}\kappa'_{j}\nabla {\mathcal{Y}_{j}}_{\varrho}(x^{*}))~\text{for all}~\varrho\in[0, 1],\\
&\kappa_{0}\geq0, \kappa'_{j}\geq0,~j\in \Lambda(x^{*}).
\end{cases}
\end{equation*}
There must be $\kappa_{0}>0$, otherwise $\{\nabla \mathcal{Y}_{j}(x^{*}): j\in \Lambda(x^{*})\}$ is not linearly independent.

Consider $\kappa_{j}=\frac{\kappa'_{j}}{\kappa_{0}}$. Then, $\kappa_{j}\geq 0$, $j\in \Lambda(x^{*})$ and
$$(0, 0, \ldots, 0)_{n}^{\top}\in (\nabla \mathcal{H}_{\varrho}(x^{*})+\sum_{j\in \Lambda(x^{*})}\kappa_{j}\nabla {\mathcal{Y}_{j}}_{\varrho}(x^{*}))~\text{for every}~\varrho\in[0, 1].$$
Since $\mathcal{Y}_{j}(x^{*})=\tilde{0}, j\in \Lambda(x^{*})$, then $\kappa_{j}{\mathcal{Y}_{j}}_{\varrho}(x^{*})=\tilde{0}$ for each $\varrho\in[0, 1]$. If $\mathcal{Y}_{j}$ ($j\notin \Lambda(x^{*})$) is gH-differentiable at $x^{*}$, let $\kappa_{j}=0$, $j\notin \Lambda(x^{*})$, the desired conclusion is obtained.
\end{proof}

\begin{example}
Consider the FOP:
\begin{equation*}
\begin{split}
\min~~&\mathcal{H}(x_{1}, x_{2})=\langle-4, -2, 0\rangle x_{1}^{2}+\langle1, 2, 3\rangle x_{2}+\langle-2, 0, 5\rangle x_{2}^{2}\\
&~~~~~~~~~~~~~~~~~~~~+\langle3, 5, 6\rangle x_{1}^{2}x_{2},\\
\text{subject to}~~&\mathcal{Y}_{1}(x_{1}, x_{2})=\langle-2, 0, 3\rangle x_{1}+\langle-6, -5, -3\rangle x_{2}\\
&~~~~~~~~~~~~~~~~~~~~\ominus_{gH}\langle-12, -10, -6\rangle\preceq\tilde{0},\\
&\mathcal{Y}_{2}(x_{1}, x_{2})=\langle3, 5, 6\rangle x_{1}+\langle-8, -7, -5\rangle x_{2}\\
&~~~~~~~~~~~~~~~~~~~~\ominus_{gH}\langle-2, -1, 0\rangle\preceq\tilde{0}.
\end{split}
\end{equation*}
Obviously, $\mathcal{H}, \mathcal{Y}_{1}$ and $\mathcal{Y}_{2}$ are gH-differentiable on $\mathbb{R}^{2}$, $x^{*}=(0, 2)\in \mathcal{O}$ and
$$\mathcal{Y}_{1}(x^{*})=\tilde{0}~\text{and}~\mathcal{Y}_{2}(x^{*})=\langle-14, -13, -10\rangle.$$
So, $\Lambda(x^{*})=\{1\}$, and we get
\begin{equation*}
\begin{split}
\nabla \mathcal{H}_{\varrho}(x^{*})&=(D_{1}\mathcal{H}_{\varrho}(x^{*}), D_{2}\mathcal{H}_{\varrho}(x^{*}))^{\top}\\
&=([0, 0], [-7+9\varrho, 23-21\varrho])^{\top},\\
\nabla {\mathcal{Y}_{1}}_{\varrho}(x^{*})&=(D_{1}{\mathcal{Y}_{1}}_{\varrho}(x^{*}), D_{2}{\mathcal{Y}_{1}}_{\varrho}(x^{*}))^{\top}\\
&=([-2+2\varrho, 3-3\varrho], [-6+\varrho, -3-2\varrho])^{\top}.
\end{split}
\end{equation*}
Taking $\kappa_{0}=2.5, \kappa_{1}=1$ and $\kappa_{2}=0$, we get the result of Theorem \ref{thm3.7}.
Taking $\kappa_{0}=1, \kappa_{1}=0.4$ and $\kappa_{2}=0$, the result of Theorem \ref{thm3.8} is obtained.
\end{example}

\begin{theorem}\label{thm3.9}(Fuzzy KKT sufficient condition).
Let $\mathcal{H}: \mathbb{T} \rightarrow\mathcal{F}_{\mathcal{C}}$ and $\mathcal{Y}_{j}: \mathbb{T} \rightarrow\mathcal{F}_{\mathcal{C}}, j=1, 2, \ldots, s$, be gH-differentiable convex fuzzy functions on the open set $\mathbb{T}$. If there exist $\kappa_{1}, \kappa_{2}, \ldots, \kappa_{s}\in\mathbb{R}$ and $x^{*}\in\mathcal{O}$ such that
\begin{equation*}
\begin{cases}
&(0, 0, \ldots, 0)_{n}^{\top}\in (\nabla \mathcal{H}_{\varrho}(x^{*})+\sum_{j=1}^{s}\kappa_{j}\nabla {\mathcal{Y}_{j}}_{\varrho}(x^{*}))~\text{for all}~\varrho\in[0, 1],\\
&\kappa_{j}\mathcal{Y}_{j}(x^{*})=\tilde{0}, j=1, 2, \ldots, s,\\
&\kappa_{0}\geq0, \kappa_{j}\geq0, j=1, 2, \ldots, s.
\end{cases}
\end{equation*}
Then, $x^{*}$ is an optimal solution of (FOP).
\end{theorem}
\begin{proof}
Since $(0, 0, \ldots, 0)_{n}^{\top}\in (\nabla \mathcal{H}_{\varrho}(x^{*})+\sum_{j=1}^{s}\kappa_{j}\nabla {\mathcal{Y}_{j}}_{\varrho}(x^{*}))$~for all~$\varrho\in[0, 1]$. Then, for any $x\in\mathbb{T}$,
$$0\in (\nabla \mathcal{H}_{\varrho}(x^{*})+\sum_{j=1}^{s}\kappa_{j}\nabla{\mathcal{Y}_{j}}_{\varrho}(x^{*}))^{\top}(x-x^{*}).$$
From $\mathcal{H}$ and $\mathcal{Y}_{j}, j=1, 2, \ldots, s$, are gH-differentiable and convex knows that
\begin{equation*}
\begin{split}
&~~~~(\nabla \mathcal{H}_{\varrho}(x^{*})+\sum_{j=1}^{s}\kappa_{j}\nabla{\mathcal{Y}_{j}}_{\varrho}(x^{*}))^{\top}(x-x^{*})\\
&=\nabla \mathcal{H}_{\varrho}(x^{*})^{\top}(x-x^{*})+\sum_{j=1}^{s}\kappa_{j}\nabla {\mathcal{Y}_{j}}_{\varrho}(x^{*})^{\top}(x-x^{*})\\
&\preceq_{LU}(\mathcal{H}_{\varrho}(x)\ominus_{gH}\mathcal{H}_{\varrho}(x^{*}))+\sum_{j=1}^{s}\kappa_{j}{(\mathcal{Y}_{j}}_{\varrho}(x)\ominus_{gH}{\mathcal{Y}_{j}}_{\varrho}(x^{*}))\\
&\preceq_{LU}(\mathcal{H}_{\varrho}(x)\ominus_{gH}\mathcal{H}_{\varrho}(x^{*})).
\end{split}
\end{equation*}
Therefore, either $[0, 0]\preceq_{LU} \mathcal{H}_{\varrho}(x)\ominus_{gH}\mathcal{H}_{\varrho}(x^{*})$ or $0\in (\mathcal{H}_{\varrho}(x)\ominus_{gH}\mathcal{H}_{\varrho}(x^{*}))$ for each $\varrho\in[0, 1]$. In any situation, $x^{*}$ be an optimal solution of (FOP). Thus, the conclusion follows.
\end{proof}

\section{Application to classification of fuzzy data}
\label{sec:4}
Consider data set: $\mathbb{D}=\{(x_{i}, y_{i}): x_{i}\in \mathbb{R}^{n}, y_{i}\in\{-1, 1\}, i=1, 2, \ldots, s\}$ and SVMs is an effective method to classify this data set, primarily using the following optimization model:
\begin{equation}\label{haha}
\begin{cases}
\min~~~&{H}(\lambda, \ell)=\frac{1}{2}\|\lambda\|^{2},\\
\text{subject to}~~~&y_{i}(\lambda^{\top}x_{i}-\ell)\geq1,~i=1, 2, \ldots, s,
\end{cases}
\end{equation}
here $\lambda=(\lambda_{1}, \lambda_{2}, \ldots, \lambda_{n})^{\top}\in \mathbb{R}^{n}$ and $\ell\in\mathbb{R}$ represents normal vector and bias respectively. The constraint $y_{i}(\lambda^{\top}x_{i}-\ell)\geq1$ indicates the condition on which side of the separating hyperplane $\lambda^{\top}x_{i}-\ell=\pm1$ the data point lies.

Data in practical problems are often accompanied by uncertainty. For example, to determine whether it will rain next weekend, we will look at the temperature, humidity and wind speed, and get the result that the temperature is about $33 {^{\circ}}$C, the humidity is about $45\%$, and the wind speed is about $12$ km/hr. For such data, people tend to estimate and truncate to get an accurate value, and with the accumulation of truncation error and rounding error, this may lead to deviations from the actual results or even the exact opposite. Fuzzy numbers are an effective tool for solving such uncertain data, and their application has gained a great deal of interest from researchers in recent years. Therefore, our next main objective is to investigate the problem of how fuzzy number data can be classified. Clearly, model \eqref{haha} is not suitable for this kind of data because these are fuzzy data. Therefore, we modify the traditional SVM problem for fuzzy data set
$$\{(\mathcal{U}_{i}, y_{i}): \mathcal{U}_{i}\in\mathcal{F}_{\mathcal{C}}^{n}, y_{i}\in\{-1, 1\}, i=1, 2, \ldots, s\}$$
through
\begin{equation}\label{hehe}
\begin{cases}
\min~~~&\mathcal{H}(\lambda, \ell)=\frac{1}{2}\|\lambda\|^{2},\\
\text{subject to}~~~&\mathcal{Y}_{i}(\lambda, \ell)=\tilde{1}\ominus_{gH}y_{i}(\lambda^{\top}\mathcal{U}_{i}\ominus_{gH}\ell)\preceq\tilde{0},\\
&~~~~~~~~~~i=1, 2, \ldots, s,
\end{cases}
\end{equation}
here we assume that components of $\mathcal{U}_{i}$'s are fuzzy numbers, and the data set consisting the core points of the fuzzy data is linearly separable.

We note that the gradients of the functions $\mathcal{H}$ and $\mathcal{Y}_{i}$ in \eqref{hehe} are
\begin{equation*}
\begin{split}
&\nabla\mathcal{H}(\lambda, \ell)=(D_{1}\mathcal{H}(\lambda, \ell), D_{2}\mathcal{H}(\lambda, \ell))^{\top}=(\lambda, \tilde{0})^{\top},\\
&\nabla\mathcal{Y}_{i}(\lambda, \ell)=(D_{1}\mathcal{Y}_{i}(\lambda, \ell), D_{2}\mathcal{Y}_{i}(\lambda, \ell))^{\top}=(-y_{i}\mathcal{U}_{i}, -y_{i})^{\top},
\end{split}
\end{equation*}
for all $i=1, 2, \ldots, s$, where $D_{1}$ and $D_{2}$ represent the gH-partial derivatives of $\lambda$ and $\ell$, respectively.

By Theorem \ref{thm3.8},  for an optimal solution $(\lambda^{*}, \ell^{*})$ of \eqref{hehe} there are $\kappa_{1}, \kappa_{2}, \ldots, \kappa_{s}\in\mathbb{R}$ and $\kappa_{1}, \kappa_{2}, \ldots, \kappa_{s}\geq0$ that make
\begin{equation}\label{con27}
(0, 0, \ldots, 0)_{n+1}^{\top}\in((\lambda^{*}, 0)^{\top}+\sum_{i=1}^{s}\kappa_{i}(-y_{i}{\mathcal{U}_{i}}_{\varrho}, -y_{i})^{\top})~\text{for all}~\varrho\in[0, 1],
\end{equation}
and
\begin{equation}\label{con28}
\tilde{0}=\kappa_{i}\mathcal{Y}_{i}(\lambda^{*}, \ell^{*}), i=1, 2, \ldots, s.
\end{equation}
\eqref{con27} may be split into
\begin{equation*}
(0, 0, \ldots, 0)_{n}^{\top}\in(\lambda^{*}+\sum_{i=1}^{m}(-\kappa_{i}y_{i}){\mathcal{U}_{i}}_{\varrho})~\text{for all}~\varrho\in[0, 1],
\end{equation*}
and
$$\sum_{i=1}^{s}\kappa_{i}y_{i}=0.$$
As follows from Theorems \ref{thm3.8} and \ref{thm3.9}, we obtain that the set of conditions for the optimal solution of \eqref{hehe} are
\begin{equation}\label{2999}
\begin{cases}
&(0, 0, \ldots, 0)_{n}^{\top}\in(\lambda^{*}+\sum_{i=1}^{s}(-\kappa_{i}y_{i}){\mathcal{U}_{i}}_{\varrho}) ~\text{for all}~\varrho\in[0, 1],\\
&\sum_{i=1}^{s}\kappa_{i}y_{i}=0,\\
&\kappa_{i}\mathcal{Y}_{i}(\lambda^{*}, \ell^{*})=\tilde{0}, i=1, 2, \ldots, s.
\end{cases}
\end{equation}
The data points $\mathcal{U}_{i}$ with $\kappa_{i}\neq0$ are known as fuzzy support vectors. It can be observed from \eqref{con28} for every $\kappa_{i}>0$, we can get $\mathcal{Y}_{i}(\lambda^{*}, \ell^{*})=\tilde{0}$.

Notice that for each $i$ and a given $\lambda^{*}$, we can get a $\ell_{i}$ from $\mathcal{Y}_{i}(\lambda, \ell)=\tilde{0}$, and we consider
\begin{equation}\label{3000}
\ell^{*}=\bigwedge_{\forall i, k_{i}>0}\ell_{i}
\end{equation}
as the bias corresponding to a given $\lambda^{*}$. We employ here the intersection operator ``$\bigwedge$" because $\ell^{*}$ must be satisfied with $\mathcal{Y}_{i}(\lambda^{*}, \ell^{*})=\tilde{0}$ for every $i$'s with $k_{i}>0$. Note that as each $\ell_{i}\in\mathcal{F}_{\mathcal{C}}$, then $\ell^{*}\in\mathcal{F}_{\mathcal{C}}$. As the core points of the fuzzy data are assumed to be linearly separable, the core of $\ell^*$ is nonempty. Since each component of $\mathcal{U}_{i}\in\mathcal{F}_{\mathcal{C}}^{n}$, $i = 1, 2, \ldots, s$, geometrically we get a fuzzy point with rectangular base corresponding to each $\mathcal{U}_i$. In the top left corner of Figure  \ref{figure_svm}, a fuzzy point on the $\mathbb{R}^{2}$ plane corresponding to an $\mathcal{U}_{i} = (\mathcal{U}_{i}^{1}, \mathcal{U}_{i}^{2})$ with two components is drawn.
\begin{figure}[h]
\centering
\includegraphics[scale=0.6]{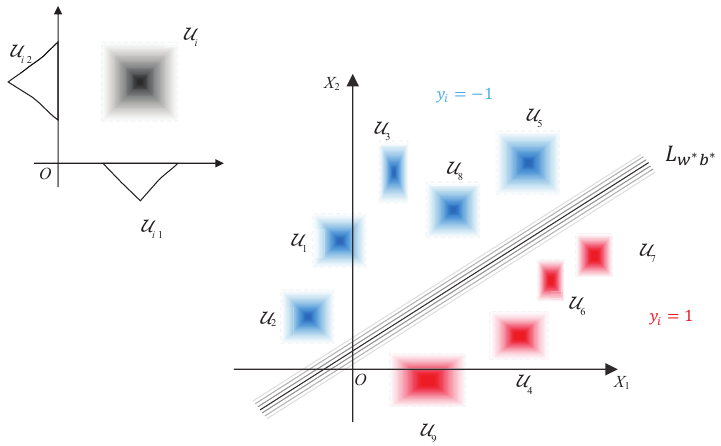}
\caption{An optimal hyperplane to classify a given fuzzy data set}
\label{figure_svm}
\end{figure}

For a given $\lambda^{*} \in \mathbb{R}$ and a fuzzy bias $\ell^{*}$ with nonempty core, we get a symmetric fuzzy line/plane $L_{\lambda^{*} \ell^{*}}$: this symmetric fuzzy line/plane $L_{\lambda^{*} \ell^{*}}$ is the union of all parallel lines/planes ${\lambda^*}^{\top} x - \ell_{\varrho} = 0$, where $\ell_{\varrho} \in \ell^{*}$ is a number in the support of $\ell^{*}$ with membership value $\varrho$. As $(\lambda^{*}, \ell^{*})$ is an optimal solution to \eqref{hehe}, we call an $L_{\lambda^{*} \ell^{*}}$ as an optimal hyperplane to classify the given fuzzy data set. In Figure \ref{figure_svm}, a possible optimal hyperplane $L_{\lambda^{*} \ell^{*}}$ to classify the given fuzzy data set of blue-colored fuzzy points with class $y_{i}=-1$ and red-colored fuzzy points with class $y_{i} = 1$.

The main procedure to effectively classify fuzzy datasets according to the fuzzy Fritz-John optimality theorem are as follows.
\begin{enumerate}
\item[\bf{Step 1}] Construct mathematical formula for fuzzy optimization problem \eqref{hehe};
\item[\bf{Step 2}] Using Theorem \ref{thm3.7}, obtain the fuzzy Fritz-John optimality condition \eqref{2999};
\item[\bf{Step 3}] Substitute the fuzzy data into the first equation in \eqref{2999} to get the value of $\lambda^{*}$;
\item[\bf{Step 4}] Take $\lambda_{i}\in\lambda^{*}$, substitute into the third equation in \eqref{2999}, and get the value of $\ell_{i}$;
\item[\bf{Step 5}] Take $\ell_{i}$ to the formula \eqref{3000} and get the value of $\ell^{*}$;
\item[\bf{Step 6}] Get the hyperplane $L_{\lambda^{*} \ell^{*}}={\lambda^{*}}^{\top}\mathcal{U}_{i}\ominus_{gH}\ell^{*}$ to see if the data set can be effectively classified, if not, go back to Step 4 and re-take $\lambda_{i}$ to solve the hyperplane; If yes, the optimal result is output.
\end{enumerate}
\begin{example}
Consider the fuzzy data set
\begin{equation*}
\begin{split}
&\mathcal{U}_{1}=[\langle2, 5, 6\rangle, \langle1, 2, 3\rangle], y_{1}=1,\\
&\mathcal{U}_{2}=[\langle3, 6, 7\rangle, \langle1, 3, 5\rangle], y_{2}=1,\\
&\mathcal{U}_{3}=[\langle4, 7, 8\rangle, \langle1, 2, 3\rangle], y_{3}=1,\\
&\mathcal{U}_{4}=[\langle0, 1, 2\rangle, \langle1, 3, 5\rangle], y_{4}=-1,\\
&\mathcal{U}_{5}=[\langle1, 2, 3\rangle, \langle1, 3, 5\rangle], y_{5}=-1,\\
&\mathcal{U}_{6}=[\langle0, 2, 3\rangle, \langle2, 5, 6\rangle], y_{6}=-1.
\end{split}
\end{equation*}

We will utilize the FOP SVM \eqref{hehe} to derive a classification hyperplane for the above data set, that is, we require discovering the possible solutions $(\lambda, \ell)$ of \eqref{2999} in conjunction with the corresponding $\kappa_{i}$'s.

We note that $\sum_{i=1}^{6}\kappa_{i}y_{i}=0$ when $(\kappa_{1}, \kappa_{2}, \kappa_{3}, \kappa_{4}, \kappa_{5}, \kappa_{6})=(1, 0, 0, 0, 1, 0)$, and for all~$\varrho\in[0, 1]$, the first condition in \eqref{2999} simplifies to
\begin{equation}\label{3131}
\begin{split}
&(0, 0)^{\top}\in \lambda+(-1){\mathcal{U}_{1}}_{\varrho}+{\mathcal{U}_{5}}_{\varrho},\\
\text{or,}~&\lambda\in (-1)((-1){\mathcal{U}_{1}}_{\varrho}+{\mathcal{U}_{5}}_{\varrho}),\\
\text{or,}~&\lambda\in ([4\varrho-1, 5-2\varrho], [3\varrho-2, 4-3\varrho]).
\end{split}
\end{equation}
Denote $\lambda=(\lambda_{1}, \lambda_{2})^{\top}\in \mathbb{R}^{2}$, the condition \eqref{3131} is simplified as
\begin{equation*}\label{3232}
4\varrho-1\leq \lambda_{1}\leq 5-2\varrho ~\text{and}~3\varrho-2\leq \lambda_{2}\leq 4-3\varrho~~\text{for all}~\varrho\in[0, 1].
\end{equation*}
Let we select $\lambda_{1}^{*}=3$ and $\lambda_{2}^{*}=1$, the third condition in \eqref{2999} and \eqref{3000} yields the set of probable values of the bias $\ell$ as
\begin{equation*}
\begin{split}
&\bigwedge_{i: \kappa_{i}>0}\{\ell: \mathcal{Y}_{i}(\lambda, \ell)=\tilde{0}\}\\
=&\{\ell\in\mathbb{R}: \mathcal{Y}_{1}(\lambda^{*}, \ell)=\tilde{0}\}\wedge\{\ell\in\mathbb{R}: \mathcal{Y}_{5}(\lambda^{*}, \ell)=\tilde{0}\}\\
=&\{\ell\in\mathbb{R}: \ell\in[6+10\varrho, 20-4\varrho], ~\forall \varrho\in[0, 1]\}\\
&~\wedge\{\ell\in\mathbb{R}: \ell\in[5+5\varrho, 15-5\varrho], ~\forall \varrho\in[0, 1]\}\\
=&\{\ell\in\mathbb{R}: \ell\in[6+10\varrho, 15-5\varrho], ~\forall \varrho\in[0, \tfrac{9}{15}]\}.
\end{split}
\end{equation*}
Consequently, which corresponds with $\lambda^{*}=(3, 1)^{\top}$, the collection of classifying hyperplanes is governed from
$$3\mathcal{U}_{1}+\mathcal{U}_{2}\ominus_{gH}\ell=\tilde{0},~\ell\in\{\ell\in\mathbb{R}: \ell\in[6+10\varrho, 15-5\varrho], ~ \varrho\in[0, \tfrac{9}{15}]\}.$$
It is worth noting that the value of the objective function $\mathcal{H}$ is equal to $5$ regardless of the choice of
$\ell$ in $\{\ell\in\mathbb{R}: \ell\in[6+10\varrho, 15-5\varrho], ~\varrho\in[0, \tfrac{9}{15}]\}$.
\end{example}

\section{Conclusions}\label{sec:5}
The major goal of this paper has been studying the optimality conditions for FOPs. Firstly, the descent direction cone of the fuzzy function has been  defined, and we have used it to derive first-order optimality conditions with the help of cone of feasible directions. Furthermore, we have extended the Gordan's theorem for the systems of fuzzy inequalities and used it to derive the necessary optimality conditions of Fritz-John and KKT for FOPs. Finally, we have applied the results obtained in this paper to a simple linearly separable SVM problem of fuzzy data sets. In the following, we will proceed to study this kind of problem, hoping to apply it for nonlinear classification and SVM with soft margins.

\section*{Acknowledgement}
The work as supported by the Postgraduate Research \& Practice Innovation Program of Jiangsu Province (KYCX23\_0668). D. Ghosh acknowledges financial support of the research grants MATRICS (MTR/2021/000696) and CRG (CRG/2022/\\
001347) from SERB, India.

\end{document}